\documentclass[12pt]{article}
\usepackage{graphics,graphicx}
\usepackage{tikz}
\usepackage{amsmath}
\usepackage{indentfirst, latexsym, bm, amssymb}

\begin{document}

\newtheorem{theorem}{Theorem}[section]
\newtheorem{corollary}[theorem]{Corollary}
\newtheorem{definition}[theorem]{Definition}
\newtheorem{conjecture}[theorem]{Conjecture}
\newtheorem{question}[theorem]{Question}
\newtheorem{lemma}[theorem]{Lemma}
\newtheorem{proposition}[theorem]{Proposition}
\newtheorem{example}[theorem]{Example}
\newenvironment{proof}{\noindent {\bf
Proof.}}{\rule{3mm}{3mm}\par\medskip}
\newcommand{\remark}{\medskip\par\noindent {\bf Remark.~~}}
\newcommand{\pp}{{\it p.}}
\newcommand{\de}{\em}

\newcommand{\JEC}{{\it Europ. J. Combinatorics},  }
\newcommand{\JCTB}{{\it J. Combin. Theory Ser. B.}, }
\newcommand{\JCT}{{\it J. Combin. Theory}, }
\newcommand{\JGT}{{\it J. Graph Theory}, }
\newcommand{\ComHung}{{\it Combinatorica}, }
\newcommand{\DM}{{\it Discrete Math.}, }
\newcommand{\ARS}{{\it Ars Combin.}, }
\newcommand{\SIAMDM}{{\it SIAM J. Discrete Math.}, }
\newcommand{\SIAMADM}{{\it SIAM J. Algebraic Discrete Methods}, }
\newcommand{\SIAMC}{{\it SIAM J. Comput.}, }
\newcommand{\ConAMS}{{\it Contemp. Math. AMS}, }
\newcommand{\TransAMS}{{\it Trans. Amer. Math. Soc.}, }
\newcommand{\AnDM}{{\it Ann. Discrete Math.}, }
\newcommand{\NBS}{{\it J. Res. Nat. Bur. Standards} {\rm B}, }
\newcommand{\ConNum}{{\it Congr. Numer.}, }
\newcommand{\CJM}{{\it Canad. J. Math.}, }
\newcommand{\JLMS}{{\it J. London Math. Soc.}, }
\newcommand{\PLMS}{{\it Proc. London Math. Soc.}, }
\newcommand{\PAMS}{{\it Proc. Amer. Math. Soc.}, }
\newcommand{\JCMCC}{{\it J. Combin. Math. Combin. Comput.}, }
\newcommand{\GC}{{\it Graphs Combin.}, }

\title{ On signless Laplacian coefficients of unicyclic graphs with
given matching number \thanks{
This work is supported by National Natural Science Foundation of
China (No:10971137), the National Basic Research Program (973) of
China (No.2006CB805900),  and a grant of Science and Technology
Commission of Shanghai Municipality (STCSM, No: 09XD1402500).  }}
\author{Jie Zhang, Xiao-Dong Zhang\thanks{Corresponding  author ({\it E-mail address:}
xiaodong@sjtu.edu.cn)}
\\
{\small Department of Mathematics},
{\small Shanghai Jiao Tong University} \\
{\small  800 DongChuan road, Shanghai, 200240,  P.R. China}\\
 }
\maketitle
 \begin{abstract}
  Let $G$ be an unicyclic graph of order $n$ and let
  $Q_G(x)= det(xI-Q(G))=\begin{matrix} \sum_{i=1}^n (-1)^i \varphi_i x^{n-i}\end{matrix}$
  be the characteristic polynomial of the signless Laplacian matrix of a graph $G$. We
  give some transformations of $G$ which decrease all signless Laplacian coefficients
  in the set $\mathcal{G}(n,m)$. $\mathcal{G}(n,m)$ denotes
  all n-vertex unicyclic graphs with matching number $m$.
  We characterize the graphs which minimize all the signless Laplacian coefficients
  in the set $\mathcal{G}(n,m)$ with odd (resp. even) girth.
  Moreover, we find the extremal graphs which
  have minimal signless Laplacian coefficients in the set $\mathcal{G}(n)$
  of all $n$-vertex unicyclic graphs with odd (resp. even) girth.

   \end{abstract}


\vskip 0.5cm

\section{Introduction}
Let $G$ be a simple undirect unicyclic graph. $V(G)$ and $E(G)$ denote its vertex
set and edge set, respectively. For every unicyclic graph $G$, $|V(G)|=|E(G)|$.
Let $d(v_i)$ denote the degree of vertex $v_i$,
and let $D(G)=diag(d(v_1),d(v_2),\cdots,d(v_n))$ be the diagonal matrix of $G$.
Furthermore, let $A(G)$ be the adjacent matrix of $G$.
The Laplacian matrix of $G$ is $L(G)=D(G)-A(G)$,
and the Laplacian characteristic polynomial is denoted by
$L_G(x)= det(xI-L(G))=\begin{matrix} \sum_{k=1}^n (-1)^k c_k x^{n-k}\end{matrix}$.
The Laplacian coefficients $c_k(G)$ of a graph $G$ can be
expressed in terms of subtree structures of $G$ by the following
result of Kelmans and Chelnokov \cite{kelmans1974}. Let $F$ be a spanning forest
of $G$ with $k$ components $T_1, T_2, \cdots, T_s$, $T_i$ has
$|V(T_i)|$ vertices, let
$$\gamma(F)=\begin{matrix} \prod_{i=1}^k |V(T_i)| \end{matrix},$$

\begin{theorem}\label{theorem1.1}(\cite{kelmans1974})
Let $\mathcal{F}_k$ be the set of all spanning forests of $G$ with
exactly $k$ components. Then the Laplacian coefficient $c_{n-k}(G)$
is expressed by $c_{n-k}(G)=\begin{matrix} \sum_{F\in\mathcal{F}_k} \gamma(F) \end{matrix}.$
\end{theorem}

Recently, the study on the Laplacian coefficients have attracted much
attention. Mohar \cite{mohar2007} fist investigate the Laplacian coefficients of
acyclic graphs under the partial order $\preceq$. Zhang et al. \cite{zhang2009}
investigated ordering trees with diameters $3$ and $4$ by the Laplacian
coefficients. $\mbox{Ili}\acute{c}$ [13] determined the n-vertex tree of
fixed diameter which minimizes the Laplacian coefficients.
$\mbox{Ili}\acute{c}$ \cite{ilic2010} determined the n-vertex tree with given matching
number having the minimum Laplacian coefficients. He and Li \cite{he2011}
studied the ordering of all $n$-vertex trees with a perfect matching by
Laplacian coefficients. $\mbox{Ili}\acute{c}$ and
$\mbox{Ili}\acute{c}$ [12] studied the n-vertex trees with fixed pendent
vertex number and 2-degree vertex number which have minimum
Laplacian coefficients. $\mbox{Stevanovi}\acute{c}$ and
$\mbox{Ili}\acute{c}$ \cite{stevanovic2009}
investigated the Laplacian coefficients of unicyclic graphs. Tan \cite{tan2011}
characterized the determined the n-vertex unicyclic graph with given
matching number which minimizes all Laplacian coefficients. He and
Shan \cite{he2010} studied the Laplacian coefficients of bicyclic graphs.

The signless Laplacian matrix of $G$, $Q(G)=D(G)+A(G)$, which is related to $L(G)$,
has also been studied recently (see [1-5,\cite{mirzakhah2012}]).
The signless Laplacian characteristic polynomial is denoted by
$Q_G(x)= det(xI-Q(G))=\begin{matrix} \sum_{i=1}^n (-1)^i \varphi_i x^{n-i}\end{matrix}$.
Using the notation from \cite{cvetkovic2007},\cite{mirzakhah2012}, a TU-subgraph of $G$
is the spanning subgraph of $G$ whose components
are trees or odd unicyclic graphs. Assume that a TU-subgraph H of $G$ contains
$c$ odd unicyclic graphs and $s$ trees $T_1,\cdots, T_s$. The weight of H can
be expressed by $W(H)=4^c\begin{matrix} \prod_{i=1}^s n_i \end{matrix}$,
in which $n_i$ is the number of $T_i$. If $H$ contains no tree, let
$W(H)=4^c$. If $H$ is empty, in other words, $H$ does not exist, let
$W(H)=0$. The signless Laplacian coefficients $\varphi_i(G)$
can be expressed in terms of the weight of TU-subgraphs of $G$.

\begin{theorem}\label{theorem1.2}(\cite{cvetkovic2007},\cite{mirzakhah2012})
Let $G$ be a connected graph. For $\varphi_i$ as above, we have $\varphi_0=1$
and
$$\varphi_i =\sum_{H_i} W(H_i), i=1,\cdots,n;$$
where the summation runs over all TU-subgraphs $H_i$ of $G$ with $i$ edges.
\end{theorem}

From Theorem~\ref{theorem1.2},
it is obvious that for a $n$-vertex connected unicyclic graph $G$, $\varphi_1(G)=2n$.
Let $g$ be the girth of $G$, which is the length of the cycle.
if $g$ is odd, $\varphi_n(G)=4$.
When $g$ is even, $G$ is bipartite graph, then $\varphi_n(G)=0$, and
$\varphi_{n-1}(G)$ counts the number of all spanning trees of $G$, thus $\varphi_{n-1}(G)=n\cdot g$.
When $G$ is bipartite graph, $L(G)$ and $Q(G)$ have the same characteristic
polynomial, so $c_i(G)=\varphi_i(G), i=0,1,2,\cdots,n$, and
the expression of $\varphi_i$ in Theorem~\ref{theorem1.2} is equivalence
to the expression of $c_i$ in Theorem~\ref{theorem1.1}.

The eigenvalues of $L(G)$ and $Q(G)$ are denoted by $\mu_1(G)\geq\cdots\geq\mu_n(G)=0$
and $\nu_1(G)\geq\cdots\geq\nu_n(G)$, respectively. The incidence energy
of $G$, $IE(G)$ for short, is defined as $IE(G)=\sum_{i=1}^n \sqrt{\nu_i(G)}$
(see [7],[8],\cite{jooyandeh2009}).

Mirzakhah and Kiani \cite{mirzakhah2012} presented a connection between the incidence energy
and the signless Laplacian coefficients.

\begin{theorem}\label{theorem1.3}(\cite{mirzakhah2012})
Let $G$ and $G^\prime$ be two graphs of order $n$. If
$\varphi_i(G)\leq\varphi_i(G^\prime)$ for $1\leq i\leq n$,
then $IE(G)\leq IE(G^\prime)$ and $IE(G)< IE(G^\prime)$ if
$\varphi_i(G)< \varphi_i(G^\prime)$ for some $i$ holds.
\end{theorem}

Let $G$ be a graph which is not a star, let $v$ be a vertex with degree
$p+1$ in $G$, such that it is adjacent with $\{u, v_1, v_2, \cdots, v_p\}$,
where $\{v_1, v_2, \cdots, v_p\}$ are pendent vertices. The graph
$G^\prime=\sigma(G, v)$ is obtained from deleting edges $vv_1,vv_2,\cdots,vv_p$
and adding edges $uv_1,uv_2,\cdots,uv_p$.

\begin{theorem}\label{theorem1.4}(\cite{mirzakhah2012})
Let $G$ be a connected graph and $G^\prime=\sigma(G,v)$, then
$\varphi_i(G)\geq\varphi_i(G^\prime)$, for every $0\leq i\leq n$,
with equality if and only if either $i\in\{0,1,n\}$ when $G$ is
non-bipartite, or $i\in\{0,1,n-1,n\}$ otherwise.
\end{theorem}

Let $G=G_1|u : G_2|v$ be the graph obtained from two disjoint graphs
$G_1$ and $G_2$ by joining a vertex $u$ of $G_1$ and a vertex $v$ of
$G_2$ by an edge. For any graph $G$ and $v\in V(G)$, let $L_{G|v}(x)$
be the principal submatrix of $L_G(x)$ obtained by deleting the row
and column corresponding to the vertex $v$.

\begin{theorem}\label{theorem1.5}(\cite{guo2005})
If $G=G_1|u : G_2|v$, then
$L_G(x)=L_{G_1}(x)L_{G_2}(x)-L_{G_1}(x)L_{G_2|v}(x)-L_{G_2}(x)L_{G_1|u}(x)$.
\end{theorem}

\begin{theorem}\label{theorem1.6}(\cite{he2008})
If $G$ be a connected graph with $n$ vertices which consists of a subgraph $H (V(H)\geq 2)$
and $n-|V(H)|$ pendent vertices attached to a vertex $v$ in $H$, then
$L_G(x)=(x-1)^{(n-|V(H)|)}L_{H}(x)-(n-|V(H)|)x(x-1)^{(n-|V(H)|-1)}L_{H|v}(x)$.
\end{theorem}

Throughout this paper, we use the following notations. Let
$\mathcal{G}(n)$ be
the set of all unicyclic graphs of order $n$. Let $\mathcal{G}(n,m)$ be the set of
all $n$-vertex unicyclic graphs with matching number $m$.
Let $\mathcal{G}(n,m)=\mathcal{G}_{g_1}(n,m)\cup \mathcal{G}_{g_2}(n,m)$,
where $\mathcal{G}_{g_1}(n,m)$ (resp.$\mathcal{G}_{g_2}(n,m)$)
denotes the subset of $\mathcal{G}(n,m)$ with
odd (resp. even) girth.
Similarly, we can define the subsets $\mathcal{G}_{g_1}(n)$ and
$\mathcal{G}_{g_2}(n)$ of $\mathcal{G}(n)$.

Denote $M(G)$ a maximum
matching of $G$, for $G\in\mathcal{G}(n,m)$, $|M(G)|=m$. Using notations from
\cite{tan2011}, for a nonpendent edge $uv$ of $G$, $E_{uv}^u=E_{vu}^u$ denote the set of all
edges incident to $u$ except $uv$. $u$ is saturated by the edge $uv$ in $M(G)$
means $uv\in M(G)$.

Mirzakhah and Kiani in \cite{mirzakhah2012} gave some results about the signless
Laplacian coefficients of a graph $G$ and ordered unicyclic graphs
with fixed girth based on the signless Laplacian coefficients.
Motivated by this result, we characterize the graphs which have
minimum signless Laplacian coefficients in $\mathcal{G}_{g_1}(n,m)$,
$\mathcal{G}_{g_2}(n,m)$ and $\mathcal{G}_{g_1}(n)$, $\mathcal{G}_{g_2}(n)$.
Let $G_{g}(s_1,t_1; s_2,t_2, \cdots,s_{g},t_{g})$ denote the connected
unicyclic graph of order $n$ obtained from a cycle $C: u_1u_2\cdots u_{g}u_1$
by adding $s_i$ pendent paths of length $2$ and $t_i$ pendent edges
at the vertex $u_i(i=1,2,\cdots,g)$.

This paper is organized as follows: In the next section, we introduce
several transformations which simultaneously decrease all the
signless Laplacian coefficients. In Section 3,
we order all the $n$-vertex graphs in the sets $\mathcal{G}_3(s_1,t_1;s_2,t_2;s_3,t_3)$
and $\mathcal{G}_4(s_1,t_1;s_2,t_2;s_3,t_3;s_4,t_4)$ with given matching
number $m$.
In Section 4, by using the results of
Section 2 and Section 3, we characterize the extremal graphs
which have minimal signless Laplacian coefficients
in $\mathcal{G}_{g_1}(n,m)$ (resp. $\mathcal{G}_{g_2}(n,m)$), as well as incidence energy.
In Section 5, similar to the methods which are used in Section 4,
we give the extremal graphs which have minimal
signless Laplacian coefficients and incidence energy
in the set $\mathcal{G}_{g_1}(n)$ (resp. $\mathcal{G}_{g_2}(n)$).

\section{Transformations}
A pendent edge is an edge which is incident to a vertex of degree $1$.
A pendent path of length $k$ attached to $u$ is a path $uv_1v_2\cdots v_k$,
satisfying $d(v_1)=d(v_2)=\cdots=d(v_{k-1})=2, d(v_k)=1$ and $d(u)\geq 3$.
Let $N_G(v)$ denote the neighbors of $v$ in the graph $G$.

For $v\in V(G)$, the subgraph $G\backslash\{v\}$ denotes the graph
obtained by deleting the vertex $v$ and all its incident edges. Denote the
cycle of $G$ as $C$, if the girth of $G$ is $g$, write $C$ as $C_g$. Assume
$V(C_g)=\{u_1, u_2, \cdots, u_g\}$, after deleting all the edges of
the cycle, $g$ trees rooted at $u_1, u_2, \cdots, u_g$ are obtained,
denote them as $T_{u_1}^G, T_{u_2}^G, \cdots, T_{u_g}^G$.

For two vertices $u, v$ in $G$, the distance $dist_G(u,v)$ of $u, v$ equals the length
of the shortest path in $G$. Let $dist_G(v,C)$ denote the distance between the vertex $v$ and the
cycle $C$ in $G$, and $dist_G(v,C)=\mbox{min} \{dist_G(u,v): u\in C)\}$. If
$v\in T_{u_i}^G$, $dist_G(v,C)=dist_G(v,u_i)$. A branch vertex
is a vertex with degree at least $3$.

\begin{definition}\label{definition2.1}
Let $G$ be a simple connected graph with $n$ vertices, and
let $uv$ be a nonpendent edge not contained in the cycle
of $G$, let $G_{uv}$ obtained from $G$ by identifying vertices
$u$ and $v$ and add a new pendent edge $ww^\prime$ to the new
vertex $w$. (see fig.1).
\end{definition}

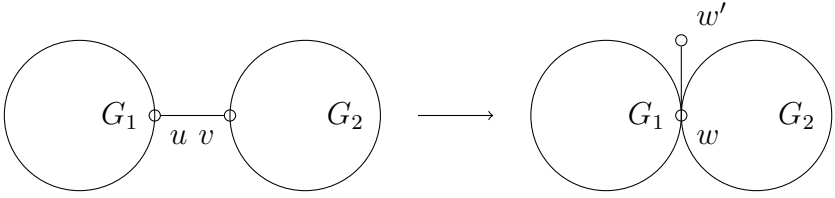
\begin{figure}
\begin{tikzpicture}
\node(v900)[label=0:$G_1$]at(-1,0){};
\node(v901)[label=0:$G_2$]at(2,0){};
\node (v0)[draw,shape=circle,inner sep=1.5pt,label=-45:$u$] at (0:0){};
\node[draw,shape=circle,inner sep=1.5pt,label=-135:$v$](v1) at (1,0){};
\draw (v0)--(v1);
\draw  (2,0) circle(1);
\draw (-1,0) circle(1);
\node(v902)[label=0:$G_1$]at(6,0){};
\node(v903)[label=0:$G_2$]at(8,0){};
\node (v0)[draw,shape=circle,inner sep=1.5pt,label=-45:$w$] at (7,0){};
\node[draw,shape=circle,inner sep=1.5pt,label=45:$w'$](v1) at (7,1){};
\draw (v0)--(v1);
\draw[->] (3.5,0)--(4.5,0);
\draw  (6,0) circle(1);
\draw (8,0) circle(1);
\end{tikzpicture}
\caption{ Transformation in Definition 2.1} \label{fig:pepper}
\end{figure}
\begin{theorem}\label{theorem2.2}
Let $G$ be a n-vertex unicyclic graph, let $G$ and $G_{uv}$ be
the two graphs presented in definition 2.1.
Then $$\varphi_i(G)\geq\varphi_i (G_{uv}),
i=0,1,\cdots,n,$$
with equality if and only if either $i\in\{0,1,n\}$ when $G$ is
non-bipartite, or $i\in\{0,1,n-1,n\}$ otherwise.
\end{theorem}
\begin{proof}
From  Theorem~\ref{theorem1.2}, according to the previous section, we have
$$\varphi_0(G)=\varphi_0(G_{uv}),
\varphi_1(G)=\varphi_1(G_{uv}), \varphi_n(G)=\varphi_n(G_{uv}).$$
When $G$ is bipartite graph, since this transformation
does not change the length of the cycle, we have $\varphi_{n-1}(G)=\varphi_{n-1}(G_{uv})$.

When $G$ is non-bipartite,
for $2\leq i\leq n-1$, denote $\mathcal{H}_i^\prime$ and $\mathcal{H}_i$
the sets of all TU-subgraphs of $G_{uv}$ and $G$ with exactly $i$ edges,
respectively. For an arbitrary TU-subgraph $H^\prime\in\mathcal{H}_i^\prime$,
let $R^\prime$ be the component of $H^\prime$ containing $w$. Let
$N_{R^\prime}(w)\cap N_G(u)=\{u_{i_1},u_{i_2},\cdots, u_{i_r}\}$, where
$0\leq r\leq \min \{d_G(u)-1, |V(R^\prime)|-1\}$,
$N_{R^\prime}(w)\cap N_G(v)=\{v_{i_1},v_{i_2},\cdots, v_{i_s}\}$, where
$0\leq s\leq \min \{d_G(v)-1, |V(R^\prime)|-1\}$.
Define $H$ with $V(H)=V(H^\prime)-\{w,w^\prime\}+\{u,v\}$, if $ww^\prime\not\in E(H^\prime)$, $E(H)=E(H^\prime)-wu_{i_1}-\cdots-wu_{i_r}-wv_{i_1}-\cdots-wv_{i_s}+uu_{i_1}+\cdots+uu_{i_r}+vv_{i_1}+\cdots+vv_{i_s}$.
If $ww^\prime\in E(H^\prime)$, let
$E(H)=E(H^\prime)-wu_{i_1}-\cdots-wu_{i_r}-wv_{i_1}-\cdots-wv_{i_s}+uu_{i_1}+
\cdots+uu_{i_r}+vv_{i_1}+\cdots+vv_{i_s}+uv-ww^\prime$.
Let $f: \mathcal{H}_i^\prime\rightarrow \mathcal{H}_i$, and
$\mathcal{H}_i^*=f(\mathcal{H}_i^\prime)=\{f(H^\prime)|H^\prime\in \mathcal{H}_i^\prime\}$.

Now we distinguish $\mathcal{H}_i^\prime$ into the following three cases.
Denote $G_1$ the connected component containing $u$ after deleting $uv$
from $G$, and let $G_2$ be the connected component
containing $v$ after deleting $uv$ from $G$.

Case 1: $ww^\prime \in H^\prime$, then $H$ and $H^\prime$
have all the components of equal size, thus $W(H)=W(H^\prime)$.

Case 2: $ww^\prime\not\in H^\prime$, $w$ is in an odd unicyclic component
$U^\prime$ of $H^\prime$,
By the symmetry of $G_1$ and $G_2$, without loss of
generality, assume $G_1$ contains an odd cycle as a subgraph.
Assume $U^\prime$ contains $a$ vertices in $G_2 \backslash \{w\}$
$(a\geq 0)$, then $W(H^\prime)=4\cdot1\cdot N$, for some constant
value $N$, $W(H)=4\cdot(a+1)\cdot N$. Thus $W(H)\geq W(H^\prime)$.

Case 3: $ww^\prime\not\in H^\prime$, $w$ is in a tree $T^\prime$ of $H^\prime$.
Assume $T^\prime$ contains $b$ vertices in $G_1\backslash\{w\}$ and
$c$ vertices in $G_2\backslash\{w\}$, then $W(H^\prime)=(b+c+1)\cdot1\cdot N$,
for some constant value $N$, $W(H)=(b+1)\cdot(c+1)\cdot N$. Thus $W(H)\geq W(H^\prime)$.

Therefore, by above discussions, $\varphi_i(G)>\varphi_i (G_{uv}),
i=2,\cdots,n-1$ holds.

When $G$ is bipartite, it is easy to prove $\varphi_i(G)>\varphi_i (G_{uv}),
i=2,\cdots,n-2$ by using above discussions of Case 1 and Case 3.
\end{proof}

\begin{remark}
When the subgraph induced by $V(G_2)$ is a star, it is easy to see
that the result of Theorem~\ref{theorem1.4} is a special case of
Theorem~\ref{theorem2.2}.

When $E_{uv}^u \cap M(G)=\emptyset$, if $uv\in M(G)$, then
$M(G_{uv})=M(G)-\{uv\}+\{ww^\prime\}$. If $vv_i\in M(G)$,
for some $v_i\in G_2$, then $M(G_{uv})=M(G)-\{vv_i\}+\{ww^\prime\}$.

When $E_{uv}^v \cap M(G)=\emptyset$, the discussion is similar.

Thus if $E_{uv}^u \cap M(G)=\emptyset$ or $E_{uv}^v \cap M(G)=\emptyset$,
we have $|M(G)|= |M(G_{uv})|$.
\end{remark}

\begin{definition}\label{definition2.3}
Let $G$ be a n-vertex unicyclic graph, and let $uv$ be a nonpendent edge of $G$
not contained in the cycle, $d_G(u)\geq 3, d_G(v)\geq 2$ and $uu^\prime$
is a pendent edge. Let $G_{uv}^\prime$ be the graph obtained by deleting vertex
$u^\prime$ and edge $uv$, identifying $u$ and $v$ to a new vertex $w$,
and adding a pendent path
$ww^\prime u^\prime$ to the new vertex $w$. (see fig.2).
\end{definition}

\begin{figure}
\begin{tikzpicture}
\node(v900)[label=0:$G_1$]at(-1,0){};
\node(v901)[label=0:$G_2$]at(2,0){};
\node (v0)[draw,shape=circle,inner sep=1.5pt,label=-45:$u$] at (0:0){};
\node (v1)[draw,shape=circle,inner sep=1.5pt,label=0:$u'$] at (0,1){};
\node[draw,shape=circle,inner sep=1.5pt,label=0:$v$](v2) at (1,0){};
\draw (v0)--(v1) (v0)--(v2);
\draw  (2,0) circle(1);
\draw (-1,0) circle(1);

\node(v902)[label=0:$G_1$]at(6,0){};
\node(v903)[label=0:$G_2$]at(8,0){};
\node (v0)[draw,shape=circle,inner sep=1.5pt,label=-45:$w$] at (7,0){};
\node[draw,shape=circle,inner sep=1.5pt,label=45:$w'$](v1) at (7,1){};
\node[draw,shape=circle,inner sep=1.5pt,label=45:$u'$](v2) at (7,2){};
\draw (v0)--(v1) (v1)--(v2);
\draw[->] (3.5,0)--(4.5,0);
\draw  (6,0) circle(1);
\draw (8,0) circle(1);
\end{tikzpicture}
\caption{ Transformation in Definition 2.3} \label{fig:pepper}
\end{figure}
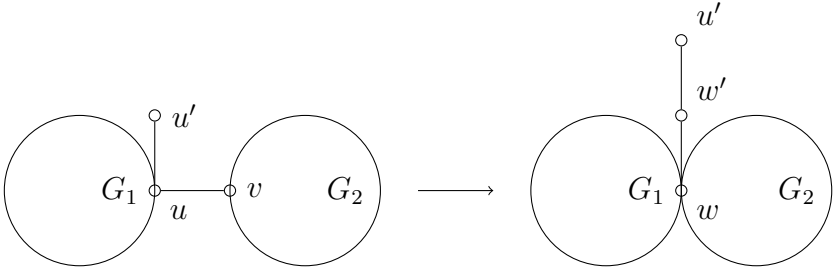

\begin{theorem}\label{theorem2.4}
Let $G$ be a n-vertex unicyclic graph, let $G$ and $G_{uv}^\prime$ be
the two graphs presented in definition 2.3.
Then $$\varphi_i(G)\geq\varphi_i (G_{uv}^\prime),
i=0,1,\cdots,n,$$
with equality if and only if either $i\in\{0,1,n\}$ when $G$ is
non-bipartite, or $i\in\{0,1,n-1,n\}$ otherwise.
\end{theorem}
\begin{proof}
For the case $\varphi_i(G)=\varphi_i (G_{uv}^\prime)$,
the proof is similar to Theorem~\ref{theorem2.2}.
Thus suppose $G$ is
non-bipartite and $2\leq i\leq n-1$, denote $\mathcal{H}_i^\prime$ and $\mathcal{H}_i$
the sets of all TU-subgraphs of $G_{uv}^\prime$ and $G$ with exactly $i$ edges,
respectively. For an arbitrary TU-subgraph $H^\prime\in\mathcal{H}_i^\prime$,
let $R^\prime$ be the component of $H^\prime$ containing $w$. Let
$N_{R^\prime}(w)\cap N_G(u)=\{u_{i_1},u_{i_2},\cdots, u_{i_r}\}$, where
$0\leq r\leq \min \{d_G(u)-2, |V(R^\prime)|-1\}$,
$N_{R^\prime}(w)\cap N_G(v)=\{v_{i_1},v_{i_2},\cdots, v_{i_s}\}$, where
$0\leq s\leq \min \{d_G(v)-1, |V(R^\prime)|-1\}$.
Define $H$ with $V(H)=V(H^\prime)-\{w,w^\prime\}+\{u,v\}$, if $ww^\prime\not\in E(H^\prime)$, $E(H)=E(H^\prime)-wu_{i_1}-\cdots-wu_{i_r}-wv_{i_1}-\cdots-wv_{i_s}+uu_{i_1}+\cdots+uu_{i_r}+vv_{i_1}+\cdots+vv_{i_s}$.
If $ww^\prime\in E(H^\prime)$, let
$E(H)=E(H^\prime)-wu_{i_1}-\cdots-wu_{i_r}-wv_{i_1}-\cdots-wv_{i_s}+uu_{i_1}+
\cdots+uu_{i_r}+vv_{i_1}+\cdots+vv_{i_s}+uv-ww^\prime$.
Let $f: \mathcal{H}_i^\prime\rightarrow \mathcal{H}_i$, and
$\mathcal{H}_i^*=f(\mathcal{H}_i^\prime)=\{f(H^\prime)|H^\prime\in \mathcal{H}_i^\prime\}$.

If we include $w, w^\prime$ in a component of $H^\prime$,
we have components of equal size
in both TU-subgraphs ($H$ and $H^\prime$), then $W(H)=W(H^\prime)$.
Now we can assume $w, w^\prime$ belong to two components of $H^\prime$
and distinguish $\mathcal{H}_i^\prime$ into the following three cases.
Denote $G_1$ the subgraph containing $u$ obtained by deleting $uv$ and $uu^\prime$ from $G$,
and let $G_2$ be the connected component
containing $v$ after deleting $uv$ from $G$.

Case 1: $w$ is in an odd unicyclic component $U^\prime$ of $H^\prime$ and
the cycle is a subgraph of $G_1$.

Subcase 1.1: $w^\prime u^\prime\not\in H^\prime$, assume
$U^\prime$ contains $b_1$ vertices in the subgraph $G_2\backslash\{w\},(b_1\geq 0)$,
then $W(H^\prime)=4\cdot1\cdot1\cdot N$, for some constant value $N$.
$W(H)=4\cdot(b_1+1)\cdot1\cdot N$, then $W(H)\geq W(H^\prime)$. Denote this set of
$H^\prime$ as $\mathcal{H}_{11}^\prime$.

Subcase 1.2: $w^\prime u^\prime\in H^\prime$, assume
$U^\prime$ contains $b_2$ vertices in the subgraph $G_2\backslash\{w\},(b_2\geq 0)$,
then $W(H^\prime)=4\cdot2\cdot N$, for some constant value $N$.
$W(H)=4\cdot(b_2+1)\cdot N$, then when $b_2\geq1$, $W(H)\geq W(H^\prime)$.
When $b_2=0$, $W(H)<W(H^\prime)$, denote this set of
$H^\prime$ as $\mathcal{H}_{12}^\prime$.

Since the subgraph induced by $V(G_2)$ is a tree, and $v$ is not a pendent
vertex, $d(v)\geq2$. Suppose one neighbor vertex of $v$ except $u$
is $u_1$, $u_1\in T^\prime\subset H^\prime$ and $|T^\prime|=c\geq1$. Then
for every TU-subgraph $H_2^\prime\in \mathcal{H}_{12}^\prime$,
let $H_1^\prime=H_2^\prime-w^\prime u^\prime+wu_1, H_1=H_2-uu^\prime+vu_1$,
it is obvious that $H_1^\prime \in\mathcal{H}_{11}^\prime$ and
the mapping $f_1: \mathcal{H}_{12}^\prime\rightarrow \mathcal{H}_{11}^\prime$ is an injection.
Then
$W(H_2^\prime)=4\cdot2\cdot c\cdot N_1$, for some constant value $N_1$.
$W(H_2)=4\cdot(1+c)\cdot N_1$, and $W(H_1^\prime)=4\cdot1\cdot 1\cdot N_1$,
$W(H_1)=4\cdot1\cdot(1+c)\cdot N_1$. Therefore,
$$\sum_{H_1^\prime \in\mathcal{H}_{11}^\prime}[W(H_1)-W(H_1^\prime)]+
\sum_{H_2^\prime\in \mathcal{H}_{12}^\prime} [W(H_2)-W(H_2^\prime)]$$$$
\geq\sum_{H_2^\prime\in \mathcal{H}_{12}^\prime}(4+4c-8c)\cdot N_1+\sum_{H_1^\prime\in f_1(\mathcal{H}_{12}^\prime)}(4c)\cdot N_1$$
$$=4\sum_{H_2^\prime\in \mathcal{H}_{12}^\prime}N_1 >0$$

Case 2: $w$ is in an odd unicyclic component $U^\prime$ of $H^\prime$ and
the cycle is a subgraph of $G_2$.

Subcase 2.1: $w^\prime u^\prime\not\in H^\prime$, assume
$U^\prime$ contains $a_1$ vertices in the subgraph $G_1\backslash\{w\},(a_1\geq 0)$,
then $W(H^\prime)=4\cdot1\cdot1\cdot N$, for some constant value $N$.
$W(H)=4\cdot(a_1+1)\cdot1\cdot N$, then $W(H)\geq W(H^\prime)$. Denote this set of
$H^\prime$ as $\mathcal{H}_{21}^\prime$.

Subcase 2.2: $w^\prime u^\prime\in H^\prime$, assume
$U^\prime$ contains $a_2$ vertices in the subgraph $G_1\backslash\{w\},(a_2\geq 0)$,
then $W(H^\prime)=4\cdot2\cdot N$, for some constant value $N$.
$W(H)=4\cdot(a_2+2)\cdot N$, then $W(H)\geq W(H^\prime)$.

Case 3: $w$ is in a tree $T^\prime$ of $H^\prime$.

Subcase 3.1: $w^\prime u^\prime\not\in H^\prime$, $w\in T^\prime \subset H^\prime$ and
$T^\prime$ contains $m_1$ vertices in the subgraph $G_1\backslash\{w\},(m_1\geq 0)$,
and $n_1$ vertices in the subgraph $G_2\backslash\{w\},(n_1\geq 0)$.
Then $W(H^\prime)=(m_1+n_1+1)\cdot1\cdot1\cdot N$, for some constant value $N$.
$W(H)=(m_1+1)\cdot(n_1+1)\cdot1\cdot N$, then $W(H)\geq W(H^\prime)$. Denote this set of
$H^\prime$ as $\mathcal{H}_{31}^\prime$.

Subcase 3.2: $w^\prime u^\prime\in H^\prime$, $w\in T^\prime \subset H^\prime$ and
$T^\prime$ contains $m_2$ vertices in the subgraph $G_1\backslash\{w\},(m_2\geq 0)$,
and $n_2$ vertices in the subgraph $G_2\backslash\{w\},(n_2\geq 0)$.
Then $W(H^\prime)=(m_2+n_2+1)\cdot2\cdot N$, for some constant value $N$.
$W(H)=(m_2+2)\cdot(n_2+1)\cdot N$, then $W(H)-W(H^\prime)= m_2\cdot(n_2-1)\cdot N$.
If $m_2=0$ or $m_2>0,n_2\geq1$, $W(H)\geq W(H^\prime)$ holds.
If $m_2>0,n_2=0$, the inequality $W(H)< W(H^\prime)$ is obtained. Denote this set of
$H^\prime$ as $\mathcal{H}_{32}^\prime$.

Now we discuss the subcase in which $H^\prime\in \mathcal{H}_{32}^\prime$.
Since $v$ is not pendent vertex, $d(v)\geq2$, assume there is at least
one neighbor vertex of $w$ in $V(G_2)$, denoted by $v_1$,
which satisfies $v_1\in T_1^\prime$,
$|T_1^\prime|=p,(p\geq1)$. Denote the set of
$H^\prime$ of this kind as $\mathcal{H}_{32}^{\prime(1)}$.
For every TU-subgraph $H_1^{\prime}\in \mathcal{H}_{32}^{\prime(1)}$,
let $H_2^\prime=H_1^{\prime}-w^\prime u^\prime+wv_1,
H_2=H_1-uu^\prime+vv_1$.
It is obvious that $H_2^\prime \in\mathcal{H}_{31}^\prime$ and the mapping
$f_2: \mathcal{H}_{32}^{\prime(1)}\rightarrow \mathcal{H}_{31}^\prime,
\mathcal{H}_{32}^{(1)}\rightarrow \mathcal{H}_{31}$ is an injection. Then
$W(H_1^{\prime})=(m_1+1)\cdot2\cdot p\cdot N$, for some constant value $N$.
$W(H_1)=(m_1+2)\cdot1\cdot p\cdot N$, and
$W(H_2^\prime)=(m_1+1+p)\cdot1\cdot 1\cdot N$,
$W(H_2)=(m_1+1)\cdot1\cdot (p+1)\cdot N$. Then
$$\sum_{H_1^{\prime}\in\mathcal{H}_{32}^{\prime(1)}}[W(H_1)-W(H_1^{\prime})]
+\sum_{H_2^{\prime}\in\mathcal{H}_{31}^{\prime}}[W(H_2)-W(H_2^\prime)]$$
$$\geq\sum_{H_1^{\prime}\in\mathcal{H}_{32}^{\prime(1)}}(-p\cdot m_1)\cdot N+
\sum_{H_2^{\prime}\in f_2(\mathcal{H}_{32}^{\prime(1)})}p\cdot m_1\cdot N=0.$$

If all the neighbors of $w$ in $V(G_2)$ are in an odd unicyclic
component $U^\prime$, then denote the set of
$H^\prime$ of this kind as $\mathcal{H}_{32}^{\prime(2)}$.
For every TU-subgraph $H_1^{\prime}\in \mathcal{H}_{32}^{\prime(2)}$,
let $H_2^\prime=H_1^{\prime}-w^\prime u^\prime+wv_1,
H_2=H_1-uu^\prime+vv_1$,
it is obvious that $H_2^\prime \in\mathcal{H}_{21}^\prime$ and the mapping
$f_3: \mathcal{H}_{32}^{\prime(2)}\rightarrow \mathcal{H}_{21}^\prime,
\mathcal{H}_{32}^{(2)}\rightarrow \mathcal{H}_{21}$ is an injection.
Then
$W(H_1^{\prime})=4\cdot2\cdot (m_1+1)\cdot N$, for some constant value $N$.
$W(H_1)=4\cdot1\cdot (m_1+2)\cdot N$, and
$W(H_2^\prime)=4\cdot1\cdot 1\cdot N$,
$W(H_2)=4\cdot1\cdot(m_1+1)\cdot N$. Then
$$\sum_{H_1^{\prime}\in\mathcal{H}_{32}^{\prime(2)}}[W(H_1)-W(H_1^{\prime})]
+\sum_{H_2^{\prime}\in\mathcal{H}_{21}^{\prime}}[W(H_2)-W(H_2^\prime)]$$
$$\geq\sum_{H_1^{\prime}\in\mathcal{H}_{32}^{\prime(2)}}(-4\cdot m_1\cdot N)
+\sum_{H_2^{\prime}\in f_3(\mathcal{H}_{32}^{\prime(2)})}(4\cdot m_1\cdot N)=0$$.

It is easy to see that the correspondence from $H^\prime$ to $H$ defined above
is an injection. By summing over possible TU-subsets of $i$ edges of $G_{uv}^\prime$,
from Theorem~\ref{theorem1.1}, the results is obtained.

When $G$ is bipartite, it is easy to prove $\varphi_i(G)>\varphi_i (G_{uv}),
i=2,\cdots,n-2$ by using above discussions of Case 3.
\end{proof}

\begin{remark}
Assume $uu^\prime\in M(G)$.

If $E_{uv}^v\cap M(G)=\{vv_j\}$ for
some $v_j\in G_2$, then
$M(G_{uv}^\prime)=M(G)-\{uu^\prime,vv_j\}+\{w^\prime u^\prime, wv_j\}$.

If $E_{uv}^v\cap M(G)=\emptyset$ and all neighbors of $u$ in $G_1$
are saturated in $M(G)$, then $M(G_{uv}^\prime)=M(G)-\{uu^\prime\}+\{w^\prime u^\prime\}$.

If $E_{uv}^v\cap M(G)=\emptyset$ and there is one neighbor $u_i$ of $u$ in $G_1$
which is not saturated in $M(G)$, then
$M(G_{uv}^\prime)=M(G)-\{uu^\prime\}+\{w^\prime u^\prime, wu_i\}$.

Thus, after the transformation of Definition~\ref{definition2.3},
$|M(G_{uv}^\prime)|=|M(G)|$ or $|M(G)|+1$.
\end{remark}

\begin{definition}\label{definition2.5}
Let $G$ be a n-vertex unicyclic graph with girth $g$, $n\geq 6$, there are
only pendent paths of length at most $2$ attached to the cycle $C_g$. $u,v,w$ are
on the cycle of length at least $5$ and $u\sim v, v\sim w$. (see fig.3).
Assume $N_G(u)=\{v,u_1,u_2,\cdots\}, N_G(v)=\{u,w,v_1,v_2,\cdots\},
N_G(w)=\{v,w_1,w_2,\cdots\}$, then the graph
$$G^\prime=G-\{ww_1,ww_2,\cdots,vv_1,vv_2,\cdots\}+
\{uw_1,uw_2,\cdots,uv_1,uv_2,\cdots\}$$
\end{definition}

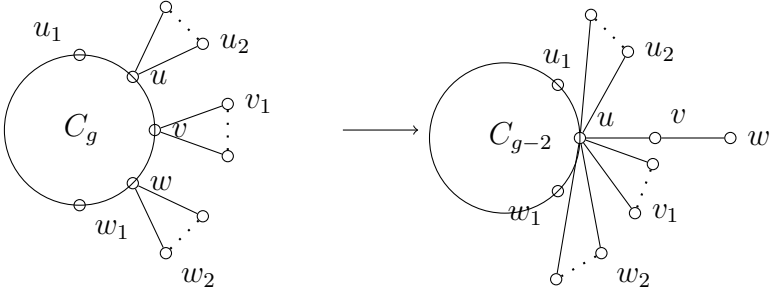
\begin{figure}
\begin{tikzpicture}
\draw  (0,0) circle(1);
\node(v900)[label=0:$C_g$]at(-0.5,0){};
\node (v)[draw,shape=circle,inner sep=1.5pt,label=0:$v$] at (0:1){};
\node [draw,shape=circle,inner sep=1.5pt,label=0:$u$](u) at (45:1){};
\node [draw,shape=circle,inner sep=1.5pt,label=150:$u_1$](u1) at (90:1){};
\node [draw,shape=circle,inner sep=1.5pt,label=-45:$w_1$](w1) at (270:1){};
\node [draw,shape=circle,inner sep=1.5pt,label=0:$w$](w) at (315:1){};
\node [draw,shape=circle,inner sep=1.5pt,label=0:$v_1$] (v1)at (10:2){};
\node [draw,shape=circle,inner sep=1.5pt,label=0:] (v2)at (-10:2){};
\node [draw,shape=circle,inner sep=1.5pt,label=0:$u_2$](u2) at (35:2){};
\node [draw,shape=circle,inner sep=1.5pt,label=0:](u3) at (55:2){};
\node [draw,shape=circle,inner sep=1.5pt,label=-45:$w_2$](w2) at (305:2){};
\node [draw,shape=circle,inner sep=1.5pt,label=0:](w3) at (325:2){};
\draw[->] (3.5,0)--(4.5,0);
\draw (v)--(v1) (v)--(v2) (u)--(u2) (u)--(u3) (w)--(w2) (w)--(w3) ;
\draw [loosely dotted,thick] (v1) -- (v2)  (u2) -- (u3) (w2) -- (w3);
\end{tikzpicture}
\begin{tikzpicture}
\draw  (0,0) circle(1);
\node(v900)[label=0:$C_{g-2}$]at(-0.5,0){};
\node (v)[draw,shape=circle,inner sep=1.5pt,label=45:$v$] at (2,0){};
\node [draw,shape=circle,inner sep=1.5pt,label=15:$u$](u) at (0:1){};
\node [draw,shape=circle,inner sep=1.5pt,label=0:$w$](w) at (3,0){};
\node [draw,shape=circle,inner sep=1.5pt,label=90:$u_1$](u1) at (45:1){};
\node [draw,shape=circle,inner sep=1.5pt,label=-135:$w_1$](w1) at (315:1){};

\node [draw,shape=circle,inner sep=1.5pt,label=0:$v_1$] (v1)at (-30:2){};
\node [draw,shape=circle,inner sep=1.5pt,label=0:] (v2)at (-10:2){};
\node [draw,shape=circle,inner sep=1.5pt,label=0:$u_2$](u2) at (35:2){};
\node [draw,shape=circle,inner sep=1.5pt,label=0:](u3) at (55:2){};
\node [draw,shape=circle,inner sep=1.5pt,label=-45:$w_2$](w2) at (-50:2){};
\node [draw,shape=circle,inner sep=1.5pt,label=0:](w3) at (-70:2){};

\draw (v)--(u) (v)--(w) (u)--(u2) (u)--(u3) (u)--(w2) (u)--(w3) (u)--(v2) (u)--(v1);
\draw [loosely dotted,thick] (v1) -- (v2)  (u2) -- (u3) (w2) -- (w3);
\end{tikzpicture}
\caption{The Transformation of Definition 2.5} \label{fig:pepper}
\end{figure}

\begin{theorem}\label{theorem2.6}
Let $G$ and $G^\prime$ be the two graphs presented in Definition~\ref{definition2.5},
and the length of the cycle of $G$ is $g$.
Then $$\varphi_i(G)\geq \varphi_i(G^\prime), i=0,1,\cdots,n,$$ with equality
if and only if $i\in\{0,1,n\}$.
\end{theorem}

\begin{proof}
For $i\in\{0,1,n\}$, the proof is similar to Theorem~\ref{theorem2.2}.
Thus suppose $2\leq i\leq n-1$, denote $\mathcal{H}_i^\prime$ and $\mathcal{H}_i$
the sets of all TU-subgraphs of $G^\prime$ and $G$ with exactly $i$ edges,
respectively.

First assume $g$ is odd.
For an arbitrary TU-subgraph $H^\prime\in\mathcal{H}_i^\prime$,
let $R^\prime$ be the component of $H^\prime$ containing $u$. Let
$N_{R^\prime}(u)\cap N_G(w)=\{w_{i_1},w_{i_2},\cdots, w_{i_r}\}$, where
$0\leq r\leq \min \{d_G(w)-1, |V(R^\prime)|-1\}$,
$N_{R^\prime}(u)\cap N_G(v)=\{v_{i_1},v_{i_2},\cdots, v_{i_s}\}$, where
$0\leq s\leq \min \{d_G(v)-2, |V(R^\prime)|-1\}$.
Define $H$ with $V(H)=V(H^\prime)$,
$E(H)=E(H^\prime)-uw_{i_1}-\cdots-uw_{i_r}-uv_{i_1}-\cdots-uv_{i_s}+ww_{i_1}+\cdots+ww_{i_r}+vv_{i_1}+\cdots+vv_{i_s}$.
Denote $U$ and $U^\prime$ be the connected component of $H$ and
$H^\prime$ containing an odd cycle, respectively.
Let $f: \mathcal{H}_i^\prime\rightarrow \mathcal{H}_i$, and
$\mathcal{H}_i^*=f(\mathcal{H}_i^\prime)=\{f(H^\prime)|H^\prime\in \mathcal{H}_i^\prime\}$.

For convenience, write $\mathcal{H}_i^\prime$ as
$\mathcal{H}^\prime$, and $\mathcal{H}_i$ as $\mathcal{H}$.
If we include $u,v,w$ in a component of $H^\prime$,
then we have components of equal sizes in both TU-subgraphs $H^\prime$
and $H$, and thus $W(H)=W(H^\prime)$ in these cases.
Denote $\mathcal{H}^{\prime(0)}=\{H^\prime|uv\in H^\prime, vw\in H^\prime\}$.
Now we can assume that $u,v,w$ belong to $2$ or $3$ components.

We distinguish $\mathcal{H}^\prime$ into the following two cases.

Case 1: $u$ is not in an odd unicyclic
component of $H^\prime$, then all components of $H^\prime$ are trees.
Assume $u\in T_1^\prime$, and there are $a_1$ vertices in
$V(T_u^G)\cap V(T_1^\prime)$ and the vertices in the counter-clockwise of $u$
(excluding $u$), and $a_3$ vertices
in the set $V(T_v^G)\cap V(T_1^\prime)$ (excluding $v$), $a_2$ vertices among the set
$V(T_w^G)\cap V(T_1^\prime)$ and the vertices in the clockwise of $u$ (excluding $w$)
$(a_1, a_2, a_3\geq0)$. Actually in this case, $a_3=s$. Denote $N$ the product of
all the orders of components of $H^\prime$ except the components containing
$u, v, w$.

Subcase 1.1: $uv\in H^\prime, vw\not\in H^\prime$,
then $W(H^\prime)=(a_1+a_2+a_3+2)\cdot1\cdot N$, for some constant value $N$.
$W(H)=(a_1+a_3+2)\cdot(a_2+1)\cdot N$, so $W(H)-W(H^\prime)=[a_2\cdot(a_1+a_3+1)]\cdot N\geq 0$.
Denote $\mathcal{H}^{\prime(11)}=\{H^\prime|u\in T_1^\prime, uv\in H^\prime, vw\not\in H^\prime, a_1=0\}$,
$\mathcal{H}^{\prime(12)}=\{H^\prime|u\in T_1^\prime, uv\in H^\prime, vw\not\in H^\prime, a_1\geq 1\}$.
Then $\sum_{H^\prime\in\mathcal{H}^{\prime(11)}}[W(H)-W(H^\prime)]\geq 0$, and
$\sum_{H^\prime\in\mathcal{H}^{\prime(12)}}[W(H)-W(H^\prime)]\geq 0.$

Subcase 1.2: $uv, vw\not\in H^\prime$,
$W(H^\prime)=(a_1+a_2+a_3+1)\cdot 1\cdot1\cdot N$, for some constant value $N$.
$W(H)=(a_1+1)\cdot(a_2+1)\cdot(a_3+1)\cdot N$,
so $W(H)-W(H^\prime)\geq 0$.
Denote $\mathcal{H}^{\prime(21)}=\{H^\prime|u\in T_1^\prime, uv, vw\not\in H^\prime, a_1=0 \mbox{or } a_2\geq1\}$,
$\mathcal{H}^{\prime(22)}=\{H^\prime|u\in T_1^\prime, uv, vw\not\in H^\prime, a_1\geq 1, a_2=0\}$.
Then $\sum_{H^\prime\in\mathcal{H}^{\prime(21)}}[W(H)-W(H^\prime)]\geq 0$, and
$\sum_{H^\prime\in\mathcal{H}^{\prime(22)}}[W(H)-W(H^\prime)]\geq 0.$

Subcase 1.3: $uv\not\in H^\prime, vw\in H^\prime$, then
$W(H^\prime)=(a_1+a_2+a_3+1)\cdot 2\cdot N$, for some constant value $N$.
$W(H)=(a_1+1)\cdot(a_2+a_3+2)\cdot N$, then $W(H)-W(H^\prime)=(a_1-1)\cdot (a_2+a_3)\cdot N$.

Denote $\mathcal{H}^{\prime(30)}=\{H^\prime|u\in T_1^\prime, uv\not\in H^\prime, vw\in H^\prime, a_1\geq1\}$,
$\mathcal{H}^{\prime(31)}=\{H^\prime|u\in T_1^\prime, uv\not\in H^\prime, vw\in H^\prime, a_1=0, a_2\geq1\}$,
$\mathcal{H}^{\prime(32)}=\{H^\prime|u\in T_1^\prime, uv\not\in H^\prime, vw\in H^\prime, a_1=a_2=0\}$.
then $$\sum_{H^\prime\in\mathcal{H}^{\prime(30)}}[W(H)-W(H^\prime)]\geq 0,$$ and
for every TU-subgraph $H^{\prime(31)}\in \mathcal{H}^{\prime(31)}$,
let $H^{\prime(11)}=H^{\prime(31)}-vw+uv,
H^{(11)}=H^{(31)}-vw+uv$,
it is obvious that $H^{\prime(11)} \in\mathcal{H}^{\prime(11)}$, and the mapping
$f_1: \mathcal{H}^{\prime(31)}\rightarrow \mathcal{H}^{\prime(11)},
\mathcal{H}^{(31)}\rightarrow \mathcal{H}^{(11)}$ is an injection. Then
$W(H^{\prime(31)})=(a_2+a_3+1)\cdot2\cdot N$, for some constant value $N$.
$W(H^{(31)})=(a_2+a_3+2)\cdot1\cdot N$, and
$W(H^{\prime(11)})=(a_2+a_3+2)\cdot1\cdot N$,
$W(H^{(11)})=(a_2+1)\cdot(a_3+2)\cdot N$. Then $$(\sum_{H^{\prime}\in
\mathcal{H}^{\prime(31)}}+\sum_{H^{\prime}\in \mathcal{H}^{\prime(11)}})[W(H)-W(H^\prime)]$$
$$=\sum_{H^{\prime(31)}\in \mathcal{H}^{\prime(31)}}[W(H^{31)})-W(H^{\prime(31)}]
+\sum_{H^{\prime(11)}\in \mathcal{H}^{\prime(11)}}[W(H^{(11)})-W(H^{\prime(11)})]$$
$$\geq\sum_{H^{\prime(31)}\in \mathcal{H}^{\prime(31)}}[-(a_2+a_3)\cdot N]+
\sum_{H^{\prime(11)}\in f_1(\mathcal{H}^{\prime(31)})}[(a_2\cdot a_3+a_2)\cdot N]$$
$$=\sum_{H^{\prime(31)}\in \mathcal{H}^{\prime(31)}}[(a_2-1)\cdot a_3\cdot N]\geq 0.$$

For every TU-subgraph $H^{\prime(32)}\in \mathcal{H}^{\prime(32)}$,
$uu_1\not\in H^{\prime(32)}$ and assume $u_1$ is in a tree of order
$c$ in $H^{\prime(32)}$. Let $H^{\prime(22)}=H^{\prime(32)}-vw+uu_1,
H^{(22)}=H^{(32)}-vw+uu_1$,
it is obvious that $H^{\prime(22)} \in\mathcal{H}^{\prime(22)}$ and the mapping
$f_2: \mathcal{H}^{\prime(32)}\rightarrow \mathcal{H}^{\prime(22)},
\mathcal{H}^{(32)}\rightarrow \mathcal{H}^{(22)}$ is an injection.
Then
$W(H^{\prime(32)})=(a_3+1)\cdot2\cdot c\cdot N$, for some constant value $N$.
$W(H^{(32)})=(a_3+2)\cdot1\cdot c\cdot N$, and
$W(H^{\prime(22)})=(a_3+1+c)\cdot 1\cdot1\cdot N$,
$W(H^{(22)})=(a_3+1)\cdot(1+c)\cdot1\cdot N$. Then
$$(\sum_{H^{\prime}\in
\mathcal{H}^{\prime(32)}}+\sum_{H^{\prime}\in \mathcal{H}^{\prime(22)}})[W(H)-W(H^\prime)]$$
$$=\sum_{H^{\prime(32)}\in \mathcal{H}^{\prime(32)}}[W(H^{(32)})-W(H^{\prime(32)})]
+\sum_{H^{\prime(22)} \in\mathcal{H}^{\prime(22)}}[W(H^{(22)})-W(H^{\prime(22)})]$$
$$\geq\sum_{H^{\prime(32)}\in \mathcal{H}^{\prime(32)}}(-a_3\cdot c\cdot N)
+\sum_{H^{\prime(22)} \in f_2(\mathcal{H}^{\prime(32)})}(a_3\cdot c\cdot N)=0.$$

Case 2: $u$ is in an odd unicyclic
component $U^\prime$ of $H^\prime$. Assume the number of vertices of
$U^\prime$ which are incident to the vertices of the cycle of
$G^\prime$ is $d (d\geq0)$.

Subcase 2.1: $uv\not\in H^\prime, vw\not\in H^\prime$, then
$W(H^\prime)=4\cdot1\cdot1\cdot N$, for some constant value $N$.
$W(H)\geq(g-1)\cdot1\cdot N$, so
$W(H)-W(H^\prime)\geq(g-5)\cdot N\geq0.$
Denote $\mathcal{H}^{\prime(4)}=\{H^\prime|u\in U^\prime, uv\not\in H^\prime,
vw\not\in H^\prime \}$, then
$\sum_{H^\prime\in\mathcal{H}^{\prime(4)}}[W(H)-W(H^\prime)]\geq 0$.

Subcase 2.2: $uv\not\in H^\prime, vw\in H^\prime$, then
$W(H^\prime)=4\cdot2\cdot N$, for some constant value $N$.
$W(H)\geq (g+d)\cdot N\geq g\cdot N$, then
$W(H)-W(H^\prime)\geq(g+d-8)\cdot N\geq(g-8)\cdot N.$
Denote $\mathcal{H}^{\prime(50)}=\{H^\prime|u\in U^\prime, uv\not\in H^\prime,
vw\in H^\prime, d=0, g=5 \}$, and
$\mathcal{H}^{\prime(51)}=\{H^\prime|u\in U^\prime, uv\not\in H^\prime,
vw\in H^\prime, d\geq1 \mbox{or } g\geq6 \}$ .

Subcase 2.3: $uv\in H^\prime, vw\not\in H^\prime$, then
$W(H^\prime)=4\cdot1\cdot N$, for some constant value $N$.
$W(H)\geq (g+d)\cdot N\geq g\cdot N$,
so $W(H)-W(H^\prime)\geq (g+d-4)\cdot N\geq(g-4)\cdot N.$
Denote $\mathcal{H}^{\prime(60)}=\{H^\prime|u\in U^\prime, uv\in H^\prime,
vw\not\in H^\prime, d=0, g=5 \}$, and
$\mathcal{H}^{\prime(61)}=\{H^\prime|u\in U^\prime, uv\in H^\prime,
vw\not\in H^\prime, d\geq1 \mbox{or } g\geq6 \}$.

For every $H^{\prime(61)}\in \mathcal{H}^{\prime(61)}$, set
$H^{\prime(51)}=H^{\prime(61)}-uv+vw$. It is obvious that $H^{\prime(51)}\in
\mathcal{H}^{\prime(51)}$ and
$f_3: \mathcal{H}^{\prime(61)}\rightarrow \mathcal{H}^{\prime(51)}$ is an injection.
For every $H^{\prime(51)}\in \mathcal{H}^{\prime(51)}$, set
$H^{\prime(61)}=H^{\prime(51)}-vw+uv$. It is obvious that $H^{\prime(61)}\in
\mathcal{H}^{\prime(61)}$ and the mapping
$g_3: \mathcal{H}^{\prime(51)}\rightarrow \mathcal{H}^{\prime(61)}$ is an injection.
Thus $g_3=f_3^{-1}$ and $f_3: \mathcal{H}^{\prime(61)}\rightarrow
\mathcal{H}^{\prime(51)}$ is a bijection.

Then $\begin{matrix} \sum_{H^\prime\in\mathcal{H}^{\prime(61)}}
(W(H)-W(H^\prime)) \end{matrix}+
\begin{matrix} \sum_{H^\prime\in\mathcal{H}^{\prime(51)}}
(W(H)-W(H^\prime)) \end{matrix}\geq
$$$\begin{matrix} \sum_{H^\prime\in\mathcal{H}^{\prime(61)}}
((g+d-4)\cdot N+ (g+d-8)\cdot N) \end{matrix}=
\begin{matrix} \sum_{H^\prime\in\mathcal{H}^{\prime(61)}}
(2\cdot(g+d-6)\cdot N) \end{matrix}\geq0.$$

When $H^\prime\in\mathcal{H}^{\prime(60)}$ or
$H^\prime\in\mathcal{H}^{\prime(50)}$, we have $g=5, d=0$ and TU-subgraph $H^\prime$
and $H$ for Subcase 2.2 and 2.3 have $4$ edges.

If $|E(H^\prime)|=4$, since
there is at least one pendent path attached to $u_1$ or
$w_1$ or $u$ in $G^\prime$, without loss of generality, let
$u_0\in N_{G^\prime}(u)\backslash \{v,u_1,w_1\}$.

For every $H^{\prime(60)}\in \mathcal{H}^{\prime(60)}$
and $H^{\prime(50)}\in \mathcal{H}^{\prime(50)}$, set
$H^{\prime(12)}=H^{\prime(60)}-u_1w_1+uu_0$ and
$H^{\prime(30)}=H^{\prime(50)}-u_1w_1+uu_0$. It is
obvious that $H^{\prime(12)}\in
\mathcal{H}^{\prime(12)}$, $H^{\prime(30)}\in
\mathcal{H}^{\prime(30)}$, and $h_1: \mathcal{H}^{\prime(60)}\rightarrow \mathcal{H}^{\prime(12)}$ is an injection,
$h_2: \mathcal{H}^{\prime(50)}\rightarrow \mathcal{H}^{\prime(30)}$ is an injection.
Furthermore, by the previous discussion, it is easy to prove that
$f_4: \mathcal{H}^{\prime(60)}\rightarrow \mathcal{H}^{\prime(50)}$ is a bijection,
and $\begin{matrix} (\sum_{H^\prime\in\mathcal{H}^{\prime(60)}}
+ \sum_{H^\prime\in\mathcal{H}^{\prime(50)}})
(W(H)-W(H^\prime)) \end{matrix}\geq$
$$\begin{matrix} \sum_{H^\prime\in\mathcal{H}^{\prime(60)}}
((g+d-4)\cdot N+ (g+d-8)\cdot N) \end{matrix}=
\begin{matrix} \sum_{H^\prime\in\mathcal{H}^{\prime(60)}}
((-2)\cdot N) \end{matrix}.$$

Then $\begin{matrix} (\sum_{H^\prime\in\mathcal{H}^{\prime(50)}}+
\sum_{H^\prime\in\mathcal{H}^{\prime(60)}}+
\sum_{H^\prime\in\mathcal{H}^{\prime(30)}}+
\sum_{H^\prime\in\mathcal{H}^{\prime(12)}})
(W(H)-W(H^\prime)) \end{matrix}$$$\geq
\begin{matrix} \sum_{H^\prime\in\mathcal{H}^{\prime(60)}}
[(-3+1+3+0)\cdot N]\end{matrix}> 0 $$

If $5\leq |E(H^\prime)|\leq n-1$, then there is at least one
pendent edge which belongs to $E(H^\prime)$, without loss of generality,
assume $u_0u_0^\prime\in E(H^\prime)$, where
$u_0\in N_{G^\prime}(u)\backslash \{v,u_1,w_1\}$.
Then by similar method of the above case, we
have $$\begin{matrix} (\sum_{H^\prime\in\mathcal{H}^{\prime(60)}}
+ \sum_{H^\prime\in\mathcal{H}^{\prime(50)}})
(W(H)-W(H^\prime)) \end{matrix}$$
$$\geq\begin{matrix} \sum_{H^\prime\in\mathcal{H}^{\prime(60)}}
((g+d-4)\cdot N+ (g+d-8)\cdot N) \end{matrix}$$
$$=\begin{matrix} \sum_{H^\prime\in\mathcal{H}^{\prime(60)}}
(-2)\cdot2\cdot N/2) \end{matrix}\geq0,$$
and $\begin{matrix} (\sum_{H^\prime\in\mathcal{H}^{\prime(50)}}+
\sum_{H^\prime\in\mathcal{H}^{\prime(60)}}+
\sum_{H^\prime\in\mathcal{H}^{\prime(30)}}+
\sum_{H^\prime\in\mathcal{H}^{\prime(12)}})
(W(H)-W(H^\prime)) \end{matrix}$$$\geq
\begin{matrix} \sum_{H^\prime\in\mathcal{H}^{\prime(60)}}
[(-4+4+0)\cdot N/2]\end{matrix}\geq 0.$$

Thus by summing over all possible subsets of $\mathcal{H}_i^\prime$,
$(\mathcal{H}_i^\prime=\mathcal{H}^{\prime(0)}\cup\mathcal{H}^{\prime(11)}\cup\mathcal{H}^{\prime(12)}
\cup\mathcal{H}^{\prime(21)}\cup\mathcal{H}^{\prime(22)}
\cup\mathcal{H}^{\prime(30)}\cup\mathcal{H}^{\prime(31)}\cup\mathcal{H}^{\prime(32)}
\cup\mathcal{H}^{\prime(4)}\cup\mathcal{H}^{\prime(50)}\cup\mathcal{H}^{\prime(51)}
\cup\mathcal{H}^{\prime(60)}\cup\mathcal{H}^{\prime(61)}),$
from Theorem~\ref{theorem1.2} and $f$ is an injection on the whole.
Then $$\varphi_i(G^\prime)=\sum_{H^\prime\in\mathcal{H}_i^\prime} W(H^\prime)
<\sum_{H\in\mathcal{H}_i^*} W(H)\leq\sum_{H\in\mathcal{H}_i} W(H)=\varphi_i(G)$$ holds
for $i=2,3,\cdots,n-1$.

When $g$ is even, the result
$\varphi_i(G)> \varphi_i(G^\prime)$ holds for $i=2,3,\cdots,n-1$, which can be proved
as Case 1.
\end{proof}

\begin{remark}
Assume there are $t_i$ (resp. $t_{i+1}, t_{i+2}$) pendent edges and
$s_i$ (resp. $s_{i+1}, s_{i+2}$) pendent paths of length $2$ attached
to $u$ (resp. $v, w$).

If $t_i, t_{i+1}, t_{i+2} \neq 0$, assume $uu^\prime, vv^\prime, ww^\prime\in M(G)$,
then $$M(G^\prime)=M(G)-\{uu^\prime, vv^\prime, ww^\prime\}+\{uu^\prime, vw\},$$
we have $M(G^\prime)=M(G)-1$.

If $t_i=t_{i+1}=0$, then $uv\in M(G)$, we have $M(G^\prime)=M(G)-\{uv\}+\{vw\}$, and
$M(G^\prime)= M(G)$.

If $t_i, t_{i+1}\neq 0$, $t_{i+2}=0$, assume $uu^\prime, vv^\prime\in M(G)$,
then $M(G^\prime)=M(G)-\{uu^\prime, vv^\prime\}+\{uu^\prime, vw\}$, we have
$M(G^\prime)=M(G)$. The case $t_{i+1}, t_{i+2}\neq 0$, $t_i=0$ is similar.

If $t_{i+1}\neq 0$, $t_i=t_{i+2}=0$, assume $vv^\prime\in M(G)$,
then $M(G^\prime)=M(G)-\{vv^\prime\}+\{uv^\prime, vw\}$, we have $M(G^\prime)=M(G)+1$.
\end{remark}

\begin{definition}\label{definition2.7}
Let $G$ be a n-vertex unicyclic graph with girth $g$, $n\geq 8$, there are
only pendent paths of lengths $1$ or $2$ attached to the cycle $C_g$.
$u,v,w$ are on the cycle of length at least $5$ and
there is at least one pendent edge attached to $u,v,w$, respectively. (see fig.4).
Assume $u\sim v, v\sim w$ and
$N_G(u)=\{v,u^\prime,u_2,u_3,\cdots\}, N_G(v)=\{u,w,v^\prime,v_2,\cdots\},
N_G(w)=\{v,w^\prime,w_2,w_3,\cdots\}$,
$uu^\prime, vv^\prime, ww^\prime$ are pendent edges of $G$. Then the graph
$$G^\prime=G-\{vw,ww_2,ww_3,\cdots,vv_2,vv_3,\cdots\}+
\{uw,uw_2,uw_3,\cdots,uv_2,uv_3,\cdots\}$$
\end{definition}

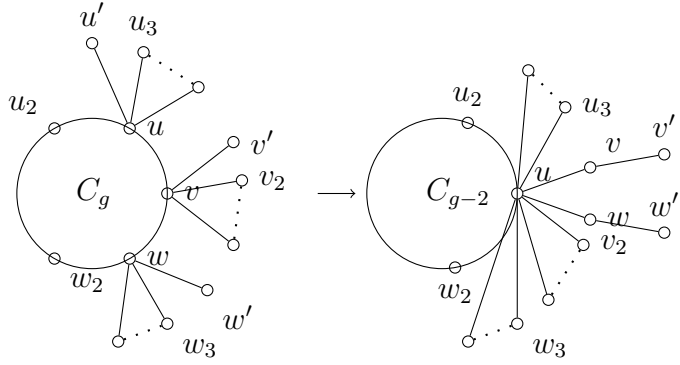
\begin{figure}
\begin{tikzpicture}
\draw  (0,0) circle(1);
\node(v900)[label=0:$C_g$]at(-0.5,0){};
\node (v)[draw,shape=circle,inner sep=1.5pt,label=0:$v$] at (0:1){};
\node [draw,shape=circle,inner sep=1.5pt,label=0:$u$](u) at (60:1){};
\node [draw,shape=circle,inner sep=1.5pt,label=150:$u_2$](u2) at (120:1){};
\node [draw,shape=circle,inner sep=1.5pt,label=90:$u_3$](u3) at (70:2){};
\node [draw,shape=circle,inner sep=1.5pt,label=90:$u'$](u1) at (90:2){};
\node [draw,shape=circle,inner sep=1.5pt,label=150:](u4) at (45:2){};
\node [draw,shape=circle,inner sep=1.5pt,label=-45:$w_2$](w2) at (240:1){};
\node [draw,shape=circle,inner sep=1.5pt,label=0:$w$](w) at (300:1){};
\node [draw,shape=circle,inner sep=1.5pt,label=-45:$w'$](w1) at (320:2){};
\node [draw,shape=circle,inner sep=1.5pt,label=-45:$w_3$](w3) at (300:2){};
\node [draw,shape=circle,inner sep=1.5pt,label=-45:](w4) at (280:2){};
\node [draw,shape=circle,inner sep=1.5pt,label=0:$v'$] (v1)at (20:2){};
\node [draw,shape=circle,inner sep=1.5pt,label=0:$v_2$] (v2)at (5:2){};
\node [draw,shape=circle,inner sep=1.5pt,label=0:] (v3)at (-20:2){};
\draw[->] (3,0)--(3.5,0);
\draw (v)--(v1) (v)--(v2) (v)--(v3) (u)--(u1) (u)--(u3)(u)--(u4) (w)--(w1) (w)--(w3)
(w)--(w4) ;
\draw [loosely dotted,thick] (v2) -- (v3)  (u4) -- (u3) (w4) -- (w3);
\end{tikzpicture}
\begin{tikzpicture}
\draw  (0,0) circle(1);
\node(v800)[label=0:$C_{g-2}$]at(-0.5,0){};
\path (4,0) coordinate (origin);
\node (v)[draw,shape=circle,inner sep=1.5pt,label=45:$v$] at (10:2){};
\node [draw,shape=circle,inner sep=1.5pt,label=15:$u$](u) at (0:1){};
\node [draw,shape=circle,inner sep=1.5pt,label=0:$w$](w) at (-10:2){};
\node [draw,shape=circle,inner sep=1.5pt,label=90:$v'$](v1) at (10:3){};
\node [draw,shape=circle,inner sep=1.5pt,label=90:$w'$](w1) at (-10:3){};

\node [draw,shape=circle,inner sep=1.5pt,label=0:$v_2$] (v2)at (-20:2){};
\node [draw,shape=circle,inner sep=1.5pt,label=0:] (v3)at (-45:2){};
\node [draw,shape=circle,inner sep=1.5pt,label=0:$u_3$](u3) at (35:2){};
\node [draw,shape=circle,inner sep=1.5pt,label=90:$u_2$](u2) at (70:1){};
\node [draw,shape=circle,inner sep=1.5pt,label=0:](u4) at (55:2){};
\node [draw,shape=circle,inner sep=1.5pt,label=-45:$w_3$](w3) at (-60:2){};
\node [draw,shape=circle,inner sep=1.5pt,label=-90:$w_2$](w2) at (280:1){};
\node [draw,shape=circle,inner sep=1.5pt,label=0:](w4) at (-80:2){};

\draw (v)--(u) (v)--(v1) (u)--(w) (w)--(w1) (u)--(u3) (u)--(u4)
(u)--(w3) (u)--(w4) (u)--(v2) (u)--(v3);
\draw [loosely dotted,thick] (v2) -- (v3)  (u3) -- (u4) (w3) -- (w4);
\end{tikzpicture}
\caption{Transformation of Definition 2.7} \label{fig:pepper}
\end{figure}

\begin{theorem}\label{theorem2.8}
Let $G$ and $G^\prime$ be the two graphs presented in Definition~\ref{definition2.7},
and the length of the cycle of $G$ is $g$, $n\geq 8$.
Then $$\varphi_i(G)\geq \varphi_i(G^\prime), i=0,1,\cdots,n,$$ with equality
if and only if $i\in\{0,1,n\}$.
\end{theorem}

\begin{proof}
For $n=8$, there is an unique graph $G$ which satisfies Theorem~\ref{theorem2.8},
by mathematical computing directly, $\varphi_i(G)>\varphi_i(G^\prime), i=2,3,\cdots,n-2$
and $\varphi_i(G)=\varphi_i(G^\prime), i=0,1,n-1,n$.

Next assume $n\geq 9$. For $i\in\{0,1,n\}$, the proof is similar to Theorem~\ref{theorem2.2}.
Thus suppose $2\leq i\leq n-1$, denote $\mathcal{H}_i^\prime$ and $\mathcal{H}_i$
the sets of all TU-subgraphs of $G^\prime$ and $G$ with exactly $i$ edges,
respectively.

First assume $g$ is odd, for an arbitrary TU-subgraph $H^\prime\in\mathcal{H}_i^\prime$,
let $R^\prime$ be the component of $H^\prime$ containing $u$. Denote
$E_1^\prime=\{uv_i:uv_i\in E(R^\prime),2\leq i\leq s\},
E_2^\prime=\{uw_i:uw_i\in E(R^\prime),2\leq i\leq t\}\bigcup\{uw:uw\in E(R^\prime)\}.
E_1=\{vv_i:uv_i\in E_1^\prime,2\leq i\leq s\},
E_2=\{ww_i:uw_i\in E_2^\prime,2\leq i\leq t\}\bigcup\{vw:uw\in E_2^\prime\}$.
Define $H$ with $V(H)=V(H^\prime)$,
$E(H)=E(H^\prime)-E_1^\prime-E_2^\prime+E_1+E_2$.
Let $f: \mathcal{H}_i^\prime\rightarrow \mathcal{H}_i$, and
$\mathcal{H}_i^*=f(\mathcal{H}_i^\prime)=\{f(H^\prime)|H^\prime\in \mathcal{H}_i^\prime\}$.

If we include $u,v,w$ in a component of $H^\prime$,
then we have components of equal sizes in both TU-subgraphs $H^\prime$
and $H$, and thus $W(H)=W(H^\prime)$ in this case.
Denote $\mathcal{H}_{(0)}^{\prime}=\{H^\prime|uv\in H^\prime, vw\in H^\prime\}$.
Now we can assume that $u,v,w$ belong to $2$ or $3$ components.

Next we distinguish $\mathcal{H}^\prime$ into the following two cases.

Case 1: $u$ is not in an odd unicyclic
component of $H^\prime$, then all components of $H^\prime$ are trees.
Assume $u\in T_1^\prime$, and $T_1^\prime$ contains $b_1 $ vertices among the set
$V(T_u^G)\backslash\{u^\prime\}$ and the vertices
in the counter-clockwise of $u$(excluding $u$), and $b_3$ vertices
in the set $V(T_v^G)\backslash\{v^\prime\}$(excluding $v$), $b_2$ vertices among the set
$V(T_w^G)\backslash\{w^\prime\}$ and the vertices in the clockwise of $u$(excluding $w$)
$(b_1, b_2, b_3\geq0)$.

Subcase 1.1: $uv\in H^\prime, vv^\prime, uw, ww^\prime\not\in H^\prime$,
then $W(H^\prime)=(b_1+b_2+b_3+2)\cdot1\cdot1\cdot1\cdot N$, for some constant value $N$.
$W(H)=(b_1+b_3+2)\cdot(b_2+1)\cdot1\cdot1\cdot N$, so $W(H)-W(H^\prime)=b_2\cdot(b_1+b_3+1)\cdot N\geq 0$.
Denote $\mathcal{H}_{1,1}^{\prime}=\{H^\prime|u\in T_1^\prime, uv\in H^\prime, vv^\prime, uw, ww^\prime\not\in H^\prime\}$.

Subcase 1.2: $uv, vv^\prime\in H^\prime, uw, ww^\prime\not\in H^\prime$,
then $W(H^\prime)=(b_1+b_2+b_3+3)\cdot1\cdot1\cdot N$, for some constant value $N$.
$W(H)=(b_1+b_3+3)\cdot(b_2+1)\cdot1\cdot N$, so $W(H)-W(H^\prime)=b_2\cdot(b_1+b_3+2)\cdot N\geq 0$.
Denote $\mathcal{H}_{1,2}^{\prime}=\{H^\prime|u\in T_1^\prime, uv,vv^\prime\in H^\prime, uw, ww^\prime\not\in H^\prime\}$.

Subcase 1.3: $uv, ww^\prime\in H^\prime, uw, vv^\prime\not\in H^\prime$,
then $W(H^\prime)=(b_1+b_2+b_3+2)\cdot1\cdot2\cdot N$, for some constant value $N$.
$W(H)=(b_1+b_3+2)\cdot(b_2+2)\cdot1\cdot N$, so $W(H)-W(H^\prime)=b_2\cdot(b_1+b_3)\cdot N\geq 0$.
Denote $\mathcal{H}_{1,3}^{\prime}=\{H^\prime|u\in T_1^\prime, uv, ww^\prime\in H^\prime, uw, vv^\prime\not\in H^\prime\}$.

Subcase 1.4: $uv, ww^\prime, vv^\prime\in H^\prime, uw\not\in H^\prime$,
then $W(H^\prime)=(b_1+b_2+b_3+3)\cdot2\cdot N$, for some constant value $N$.
$W(H)=(b_1+b_3+3)\cdot(b_2+2)\cdot N$, so $W(H)-W(H^\prime)=b_2\cdot(b_1+b_3+1)\cdot N\geq 0$.
Denote $\mathcal{H}_{1,4}^{\prime}=\{H^\prime|u\in T_1^\prime, uv, ww^\prime, vv^\prime\in H^\prime, uw\not\in H^\prime\}$.

Subcase 1.5: $uw\in H^\prime, uv, ww^\prime, vv^\prime\not\in H^\prime$,
then $W(H^\prime)=(b_1+b_2+b_3+2)\cdot1\cdot1\cdot1\cdot N$, for some constant value $N$.
$W(H)=(b_1+1)\cdot(b_2+b_3+2)\cdot1\cdot1\cdot N$, so $W(H)-W(H^\prime)=b_1\cdot(b_2+b_3+1)\cdot N\geq 0$.
Denote $\mathcal{H}_{1,5}^{\prime}=\{u\in T_1^\prime, H^\prime|uw\in H^\prime, uv, ww^\prime, vv^\prime\not\in H^\prime\}$.

Subcase 1.6: $uw, vv^\prime\in H^\prime, uv, ww^\prime\not\in H^\prime$,
then $W(H^\prime)=(b_1+b_2+b_3+2)\cdot2\cdot1\cdot N$, for some constant value $N$.
$W(H)=(b_1+1)\cdot(b_2+b_3+3)\cdot1\cdot N$, so $W(H)-W(H^\prime)=(b_1-1)\cdot(b_2+b_3+1)\cdot N$.
Denote $\mathcal{H}_{1,6}^{\prime(1)}=\{H^\prime|u\in T_1^\prime, uw, vv^\prime\in H^\prime, uv, ww^\prime\not\in H^\prime, b_1\geq1\}$.
If $H^\prime\in\mathcal{H}_{1,6}^{\prime(1)}, W(H)-W(H^\prime)\geq 0$.
Denote $\mathcal{H}_{1,6}^{\prime(2)}=\{H^\prime|u\in T_1^\prime, uw, vv^\prime\in H^\prime, uv, ww^\prime\not\in H^\prime, b_1=0\}$.
For every $H_1^\prime\in\mathcal{H}_{1,6}^{\prime(2)}$,
assume $u_2$ is in a component of $H_1^\prime$ of order $p$.
Set $H_2^\prime=H_1^\prime-vv^\prime+uu_2$, it is obvious that
$H_2^\prime\in\mathcal{H}_{1,5}^{\prime}$ in which $b_1=p\geq 1$
and $f_1: \mathcal{H}_{1,6}^{\prime(2)}\rightarrow \mathcal{H}_{1,5}^{\prime}$ is an injection.
Then $$(\sum_{H^\prime\in \mathcal{H}_{1,6}^{\prime(2)}}+
\sum_{H^\prime\in \mathcal{H}_{1,5}^{\prime}}) (W(H)-W(H^\prime))\geq
(\sum_{H^\prime\in \mathcal{H}_{1,6}^{\prime(2)}}+
\sum_{H^\prime\in f_1(\mathcal{H}_{1,6}^{\prime(2)})}) (W(H)-W(H^\prime))=0.$$

Subcase 1.7: $uw, ww^\prime\in H^\prime, uv, vv^\prime\not\in H^\prime$,
then $W(H^\prime)=(b_1+b_2+b_3+3)\cdot1\cdot1\cdot N$, for some constant value $N$.
$W(H)=(b_1+1)\cdot(b_2+b_3+3)\cdot1\cdot1\cdot N$, so $W(H)-W(H^\prime)=b_1\cdot(b_2+b_3+2)\cdot N\geq 0$.
Denote $\mathcal{H}_{1,7}^{\prime}=\{H^\prime|u\in T_1^\prime, uw, ww^\prime\in H^\prime, uv, vv^\prime\not\in H^\prime\}$.

Subcase 1.8: $uw, vv^\prime, ww^\prime\in H^\prime, uv\not\in H^\prime$,
then $W(H^\prime)=(b_1+b_2+b_3+3)\cdot2\cdot N$, for some constant value $N$.
$W(H)=(b_1+1)\cdot(b_2+b_3+4)\cdot N$, so $W(H)-W(H^\prime)=(b_1-1)\cdot(b_2+b_3+2)\cdot N$.
Denote $\mathcal{H}_{1,8}^{\prime(1)}=\{H^\prime|u\in T_1^\prime, uw, vv^\prime, ww^\prime\in H^\prime, uv\not\in H^\prime, b_1\geq1\}$.
If $H^\prime\in\mathcal{H}_{1,8}^{\prime(1)}, W(H)-W(H^\prime)\geq 0$.
Denote $\mathcal{H}_{1,8}^{\prime(2)}=\{H^\prime|u\in T_1^\prime, uw, vv^\prime, ww^\prime\in H^\prime, uv\not\in H^\prime, b_1=0\}$.
For every $H_1^\prime\in\mathcal{H}_{1,8}^{\prime(2)}$,
assume $u_2$ is in a component of $H_1^\prime$ of order $p$.
Set $H_2^\prime=H_1^\prime-vv^\prime+uu_2$, it is obvious that
$H_2^\prime\in\mathcal{H}_{1,7}^{\prime}$ in which $b_1=p\geq 1$
and $f_2: \mathcal{H}_{1,8}^{\prime(2)}\rightarrow \mathcal{H}_{1,7}^{\prime}$ is an injection.
Then $$(\sum_{H^\prime\in \mathcal{H}_{1,8}^{\prime(2)}}+
\sum_{H^\prime\in \mathcal{H}_{1,7}^{\prime}}) (W(H)-W(H^\prime))\geq
(\sum_{H^\prime\in \mathcal{H}_{1,8}^{\prime(2)}}+
\sum_{H^\prime\in f_2(\mathcal{H}_{1,8}^{\prime(2)})}) (W(H)-W(H^\prime))=0.$$

Subcase 1.9: $uw, ww^\prime, uv, vv^\prime\not\in H^\prime$,
then $W(H^\prime)=(b_1+b_2+b_3+1)\cdot1\cdot1\cdot1\cdot1\cdot N$, for some constant value $N$.
$W(H)=(b_1+1)\cdot(b_2+1)\cdot(b_3+1)\cdot1\cdot1\cdot N$,
so $W(H)-W(H^\prime)=[b_1\cdot b_2\cdot b_3+b_1\cdot b_2+b_1\cdot b_3+b_2\cdot b_3]\cdot N\geq 0$.
Denote $\mathcal{H}_{1,9}^{\prime}=\{H^\prime|u\in T_1^\prime, uw, ww^\prime, uv, vv^\prime\not\in H^\prime\}$.

Subcase 1.10: $ ww^\prime\in H^\prime, uw, uv, vv^\prime\not\in H^\prime$,
then $W(H^\prime)=(b_1+b_2+b_3+1)\cdot1\cdot1\cdot2\cdot N$, for some constant value $N$.
$W(H)=(b_1+1)\cdot(b_2+2)\cdot(b_3+1)\cdot1\cdot N$,
so $W(H)-W(H^\prime)=[b_2\cdot(b_1\cdot b_3+b_1+b_3-1)+2\cdot b_1\cdot b_3]\cdot N$.
Denote $\mathcal{H}_{1,10}^{\prime(1)}=\{H^\prime|u\in T_1^\prime, ww^\prime\in H^\prime, uw, uv, vv^\prime\not\in H^\prime, b_1+b_3\neq 0\}$.
If $H^\prime\in\mathcal{H}_{1,9}^{\prime(1)}, W(H)-W(H^\prime)\geq 0$.
Denote $\mathcal{H}_{1,10}^{\prime(2)}=\{H^\prime|u\in T_1^\prime, ww^\prime\in H^\prime, uw, uv, vv^\prime\not\in H^\prime, b_1=b_3=0\}$.
For every $H_1^\prime\in\mathcal{H}_{1,10}^{\prime(2)}$,
set $H_2^\prime=H_1^\prime-ww^\prime+uv$, it is obvious that
$H_2^\prime\in\mathcal{H}_{1,1}^{\prime}$ in which $b_1=0, b_3=0$
and $f_3: \mathcal{H}_{1,10}^{\prime(2)}\rightarrow \mathcal{H}_{1,1}^{\prime}$ is an injection.
Then $$(\sum_{H^\prime\in \mathcal{H}_{1,10}^{\prime(2)}}+
\sum_{H^\prime\in \mathcal{H}_{1,1}^{\prime}}) (W(H)-W(H^\prime))\geq
(\sum_{H^\prime\in \mathcal{H}_{1,10}^{\prime(2)}}+
\sum_{H^\prime\in f_3(\mathcal{H}_{1,10}^{\prime(2)})}) (W(H)-W(H^\prime))=0.$$

Subcase 1.11: $ vv^\prime\in H^\prime, uw, uv, ww^\prime\not\in H^\prime$,
then $W(H^\prime)=(b_1+b_2+b_3+1)\cdot1\cdot1\cdot2\cdot N$, for some constant value $N$.
$W(H)=(b_1+1)\cdot(b_2+1)\cdot(b_3+2)\cdot1\cdot N$,
so $W(H)-W(H^\prime)=[b_3\cdot(b_1\cdot b_2+b_1+b_2-1)+2\cdot b_1\cdot b_2]\cdot N$.
Denote $\mathcal{H}_{1,11}^{\prime(1)}=\{H^\prime|u\in T_1^\prime, vv^\prime\in H^\prime, uw, uv, ww^\prime\not\in H^\prime, b_1+b_2\neq 0\}$.
If $H^\prime\in\mathcal{H}_{1,11}^{\prime(1)}, W(H)-W(H^\prime)\geq 0$.
Denote $\mathcal{H}_{1,11}^{\prime(2)}=\{H^\prime|u\in T_1^\prime, vv^\prime\in H^\prime, uw, uv, ww^\prime\not\in H^\prime, b_1=b_2=0\}$.
For every $H_1^\prime\in\mathcal{H}_{1,11}^{\prime(2)}$,
assume $u_2$ is in a component of $H_1^\prime$ of order $p$.
Set $H_2^\prime=H_1^\prime-vv^\prime+uu_2$, it is obvious that
$H_2^\prime\in\mathcal{H}_{1,9}^{\prime}$ in which $b_1=p\geq 1, b_2=0$
and $f_4: \mathcal{H}_{1,11}^{\prime(2)}\rightarrow \mathcal{H}_{1,9}^{\prime}$ is an injection.
Then $$(\sum_{H^\prime\in \mathcal{H}_{1,11}^{\prime(2)}}+
\sum_{H^\prime\in \mathcal{H}_{1,9}^{\prime}}) (W(H)-W(H^\prime))\geq
(\sum_{H^\prime\in \mathcal{H}_{1,11}^{\prime(2)}}+
\sum_{H^\prime\in f_4(\mathcal{H}_{1,11}^{\prime(2)})}) (W(H)-W(H^\prime))=0.$$

Subcase 1.12: $ vv^\prime, ww^\prime\in H^\prime, uw, uv\not\in H^\prime$,
then $W(H^\prime)=(b_1+b_2+b_3+1)\cdot2\cdot2\cdot N$, for some constant value $N$.
$W(H)=(b_1+1)\cdot(b_2+2)\cdot(b_3+2)\cdot N$,
so $W(H)-W(H^\prime)=[b_1\cdot b_2\cdot b_3+b_2\cdot b_3+2\cdot (b_1-1)\cdot (b_2+b_3)]\cdot N$.
Denote $\mathcal{H}_{1,12}^{\prime(1)}=\{H^\prime|u\in T_1^\prime, vv^\prime, ww^\prime\in H^\prime,
uw, uv\not\in H^\prime, b_1\geq1 \mbox{or }b_1=b_2=b_3= 0\}$.
If $H^\prime\in\mathcal{H}_{1,11}^{\prime(1)}, W(H)-W(H^\prime)\geq 0$.
Denote $\mathcal{H}_{1,12}^{\prime(2)}=\{H^\prime|u\in T_1^\prime, vv^\prime, ww^\prime\in H^\prime,
uw, uv\not\in H^\prime, b_1=0, b_3=0, b_2\geq1\}$.
For every $H_1^\prime\in\mathcal{H}_{1,12}^{\prime(2)}$,
set $H_2^\prime=H_1^\prime-ww^\prime+uv$, it is obvious that
$H_2^\prime\in\mathcal{H}_{1,2}^{\prime}$ in which $b_1=0, b_3=0$
and $f_5: \mathcal{H}_{1,12}^{\prime(2)}\rightarrow \mathcal{H}_{1,2}^{\prime}$ is an injection.
Then $$(\sum_{H^\prime\in \mathcal{H}_{1,12}^{\prime(2)}}+
\sum_{H^\prime\in \mathcal{H}_{1,2}^{\prime}}) (W(H)-W(H^\prime))\geq
(\sum_{H^\prime\in \mathcal{H}_{1,12}^{\prime(2)}}+
\sum_{H^\prime\in f_5(\mathcal{H}_{1,12}^{\prime(2)})}) (W(H)-W(H^\prime))=0.$$
Denote $\mathcal{H}_{1,12}^{\prime(3)}=\{H^\prime|u\in T_1^\prime, vv^\prime, ww^\prime\in H^\prime,
uw, uv\not\in H^\prime, b_1=0, b_3\geq 1\}$.
For every $H_1^\prime\in\mathcal{H}_{1,12}^{\prime(3)}$,
assume $u_2$ is in a component of $H_1^\prime$ of order $p$.
Set $H_2^\prime=H_1^\prime-vv^\prime+uu_2$, it is obvious that
$H_2^\prime\in\mathcal{H}_{1,10}^{\prime(1)}$ in which $b_1=p\geq 1, b_3\geq 1$
and denote this kind of subset of $\mathcal{H}_{1,10}^{\prime(1)}$
as $\mathcal{H}_{1,10}^{\prime(11)}$. Moreover,
$f_6: \mathcal{H}_{1,12}^{\prime(3)}\rightarrow \mathcal{H}_{1,10}^{\prime(11)}$ is an injection.
Then $$(\sum_{H^\prime\in \mathcal{H}_{1,12}^{\prime(3)}}+
\sum_{H^\prime\in \mathcal{H}_{1,10}^{\prime(11)}}) (W(H)-W(H^\prime))\geq
(\sum_{H^\prime\in \mathcal{H}_{1,12}^{\prime(3)}}+
\sum_{H^\prime\in f_6(\mathcal{H}_{1,12}^{\prime(3)})}) (W(H)-W(H^\prime))\geq0.$$

Case 2: $u$ is in an odd unicyclic component $U^\prime$ of $H^\prime$.

Subcase 2.1: There is exactly one edge among $\{uv, uw\}$ which belongs
to $E(H^\prime)$. Assume $|V(U^\prime)\backslash\{v, v^\prime, w, w^\prime\}|= x$,
where $x\geq g(G^\prime)\geq 3$. If $vv^\prime, ww^\prime\not\in H^\prime$, then it is
obvious that there is a bijection between the two sets
$\mathcal{H}_{2,1}^{\prime}=\{H^\prime|u\in U^\prime, uw\in H^\prime, vv^\prime, ww^\prime, uv\not\in H^\prime \},
\mathcal{H}_{2,2}^{\prime}=\{H^\prime|u\in U^\prime, uv\in H^\prime, vv^\prime, ww^\prime, uw\not\in H^\prime \}$.
Then $$\sum_{H^\prime\in \mathcal{H}_{2,1}^{\prime}\cup\mathcal{H}_{2,2}^{\prime}} [W(H)-W(H^\prime)]
=\sum_{H^\prime\in \mathcal{H}_{2,1}^{\prime}\cup\mathcal{H}_{2,2}^{\prime}} (2\cdot x+4-8)\cdot N\geq 0.$$

If there is exactly one edge among $\{vv^\prime, ww^\prime\}$ which belongs
to $E(H^\prime)$, then it is obvious that there is a bijection between every two sets of
$\mathcal{H}_{2,3}^{\prime}=\{H^\prime|u\in U^\prime, uv, vv^\prime\in H^\prime, ww^\prime, uw\not\in H^\prime \},
\mathcal{H}_{2,4}^{\prime}=\{H^\prime|u\in U^\prime, uv, ww^\prime\in H^\prime, vv^\prime, uw\not\in H^\prime \},
\mathcal{H}_{2,5}^{\prime}=\{H^\prime|u\in U^\prime, uw, vv^\prime\in H^\prime, ww^\prime, uv\not\in H^\prime \},
\mathcal{H}_{2,6}^{\prime}=\{H^\prime|u\in U^\prime, uw, ww^\prime\in H^\prime, vv^\prime, uv\not\in H^\prime \}$.
Then $$\sum_{H^\prime\in \mathcal{H}_{2,3}^{\prime}\cup\mathcal{H}_{2,4}^{\prime}
\cup\mathcal{H}_{2,5}^{\prime}\cup\mathcal{H}_{2,6}^{\prime}} [W(H)-W(H^\prime)]
=\sum_{H^\prime\in \mathcal{H}_{2,3}^{\prime}\cup\mathcal{H}_{2,4}^{\prime}
\cup\mathcal{H}_{2,5}^{\prime}\cup\mathcal{H}_{2,6}^{\prime}} (4\cdot x+12-24)\cdot N\geq 0.$$

If $vv^\prime, ww^\prime\in H^\prime$, then it is
obvious that there is a bijection between the two sets
$\mathcal{H}_{2,7}^{\prime}=\{H^\prime|u\in U^\prime, uv, vv^\prime, ww^\prime\in H^\prime, uw\not\in H^\prime \},
\mathcal{H}_{2,8}^{\prime}=\{H^\prime|u\in U^\prime, uw, vv^\prime, ww^\prime\in H^\prime, uv\not\in H^\prime \}$.
Then $$\sum_{H^\prime\in \mathcal{H}_{2,7}^{\prime}\cup\mathcal{H}_{2,8}^{\prime}} [W(H)-W(H^\prime)]
=\sum_{H^\prime\in \mathcal{H}_{2,7}^{\prime}\cup\mathcal{H}_{2,8}^{\prime}} (2\cdot x+8-16)\cdot N.$$
When $x \geq4$, the above equation is nonnegative.
When $x=3$, then $g(G^\prime)=3$, since $uu^\prime\in E(G)$,
for every $H_1^\prime\in \mathcal{H}_{2,7}^{\prime},
H_2^\prime\in \mathcal{H}_{2,8}^{\prime}$,
set $H_3^\prime=H_1^\prime-u_1w_1+uu^\prime,
H_4^\prime=H_2^\prime-u_1w_1+uu^\prime$. It is easy to obtain
that
$H_3^\prime\in\mathcal{H}_{1,4}^{\prime}, H_4^\prime\in\mathcal{H}_{1,8}^{\prime(1)}$,
in which $b_1=2, b_2=1, b_3=0, g(G^\prime)=3$ and denote this kind of subset of
$\mathcal{H}_{1,8}^{\prime(1)}$ as $\mathcal{H}_{1,8}^{\prime(11)}$. Moreover,
$f_7: \mathcal{H}_{2,7}^{\prime}\rightarrow \mathcal{H}_{1,4}^{\prime}$ and
$f_8: \mathcal{H}_{2,8}^{\prime}\rightarrow \mathcal{H}_{1,8}^{\prime(11)}$ are injections.
Then $$\sum_{H^\prime\in \mathcal{H}_{2,7}^{\prime}\cup\mathcal{H}_{2,8}^{\prime}
\cup \mathcal{H}_{1,4}^{\prime}\cup\mathcal{H}_{1,8}^{\prime(11)}}[W(H)-W(H^\prime)]
\geq\sum_{H^\prime\in \mathcal{H}_{2,7}^{\prime}\cup\mathcal{H}_{2,8}^{\prime}
\cup f_7(\mathcal{H}_{2,7}^{\prime})\cup f_8(\mathcal{H}_{2,8}^{\prime})}(-2+2+3)>0.$$

Subcase 2.2: When $uw, uv\not\in H^\prime$, assume
$|V(U^\prime)\backslash (V(T_v^G)\cup\{u\})|=b$ and
$|V(U^\prime)\cap V(T_v^G)|=b_3$, then $|V(U^\prime)|=b+b_3+1$,
with $b\geq g(G^\prime)-1\geq 2,
b_3\geq 0$. Denote $N$ the product of the orders of all components of
$H^\prime$ except the components containing $\{u,v,w,v^\prime,w^\prime\}$.

If $vv^\prime, ww^\prime\not\in H^\prime$, then
$W(H)-W(H^\prime)=[(b_3+1)\cdot(b+2)-4]\cdot N\geq 2\cdot b_3\cdot N\geq0$.
Denote $\mathcal{H}_{2,9}^{\prime}=\{H^\prime|u\in U^\prime, uw, uv, vv^\prime, ww^\prime\not\in H^\prime\}$.

If there is exactly one edge in $\{vv^\prime, ww^\prime\}$ which belongs to $E(H^\prime)$,
then it is obvious that there is a bijection between the two sets
$\mathcal{H}_{2,10}^{\prime}=\{H^\prime|u\in U^\prime, ww^\prime\in H^\prime, vv^\prime, uv, uw\not\in H^\prime \},
\mathcal{H}_{2,11}^{\prime}=\{H^\prime|u\in U^\prime, vv^\prime\in H^\prime, ww^\prime, uv, uw\not\in H^\prime \}$.
Then $$\sum_{H^\prime\in \mathcal{H}_{2,10}^{\prime}\cup\mathcal{H}_{2,11}^{\prime}} [W(H)-W(H^\prime)]
=\sum_{H^\prime\in \mathcal{H}_{2,10}^{\prime}\cup\mathcal{H}_{2,11}^{\prime}}
((b_3+2)\cdot(b+2)+(b_3+1)\cdot(b+3)-16)\cdot N$$
$$\geq\sum_{H^\prime\in \mathcal{H}_{2,10}^{\prime}\cup\mathcal{H}_{2,11}^{\prime}}
(2\cdot b_3\cdot b+3\cdot b+5\cdot b_3-9)\cdot N.$$
When $b\geq 3$ or $b=2, b_3\geq1$, the above equation is nonnegative,
and denote this kind of subset of $\mathcal{H}_{2,10}^{\prime}, \mathcal{H}_{2,11}^{\prime}$
as $\mathcal{H}_{2,10}^{\prime(1)}, \mathcal{H}_{2,11}^{\prime(1)}$, respectively.
Denote $$\mathcal{H}_{2,10}^{\prime(2)}=\{H^\prime|u\in U^\prime, ww^\prime\in H^\prime,
vv^\prime, uv, uw\not\in H^\prime, b=2, b_3=0\},$$
$$\mathcal{H}_{2,11}^{\prime(2)}=\{H^\prime|u\in U^\prime, vv^\prime\in H^\prime,
ww^\prime, uv, uw\not\in H^\prime, b=2, b_3=0\}.$$ When
$b=2, b_3=0$, then $g(G^\prime)=3$ and $|V(U^\prime)|=3$, since the pendent
edge $uu^\prime\in E(G)$, for every $H_1^\prime\in\mathcal{H}_{2,10}^{\prime(2)},
H_2^\prime\in\mathcal{H}_{2,11}^{\prime(2)}$, set
$H_3^\prime=H_1^\prime-u_2w_2+uu^\prime, H_4^\prime=H_2^\prime-u_2w_2+uu^\prime$.
It is obvious that $H_3^\prime\in \mathcal{H}_{1,10}^{\prime(1)}$,
in which $b_1=2, b_2=1, b_3=0$, and denote this kind of subset of $\mathcal{H}_{1,10}^{\prime(1)}$
as $\mathcal{H}_{1,10}^{\prime(12)}$. $H_4^\prime\in \mathcal{H}_{1,11}^{\prime(1)}$,
in which $b_1=2, b_2=1, b_3=0$, and denote this kind of subset of $\mathcal{H}_{1,11}^{\prime(1)}$
as $\mathcal{H}_{1,11}^{\prime(11)}$. Moreover,
$f_9: \mathcal{H}_{2,10}^{\prime(2)}\rightarrow \mathcal{H}_{1,10}^{\prime(12)}$ is an injection,
$f_{10}: \mathcal{H}_{2,11}^{\prime(2)}\rightarrow \mathcal{H}_{1,11}^{\prime(11)}$ is an injection.
Then $$\sum_{H^\prime\in \mathcal{H}_{2,10}^{\prime(2)}\cup\mathcal{H}_{2,11}^{\prime(2)}\cup
\mathcal{H}_{1,10}^{\prime(12)}\cup\mathcal{H}_{1,11}^{\prime(11)}} [W(H)-W(H^\prime)]$$
$$\geq \sum_{H^\prime\in \mathcal{H}_{2,10}^{\prime(2)}\cup\mathcal{H}_{2,11}^{\prime(2)}\cup
f_9(\mathcal{H}_{2,10}^{\prime(2)})\cup f_{10}(\mathcal{H}_{2,11}^{\prime(2)})} (-3+4+1)\cdot N >0.$$

If $vv^\prime, ww^\prime\in H^\prime$, then
$W(H)-W(H^\prime)=[(b_3+2)\cdot(b+3)-16]\cdot N=[b_3\cdot b+3\cdot b_3+2\cdot b-10]\cdot N$.
When $b=2, b_3\geq2$ or $b=3, b_3\geq1$ or $b=4, b_3\geq1$ or $b\geq5$, the
above equation is nonnegative.

When $b=2, b_3=0$, then $g(G^\prime)=3$ and $|V(U^\prime)|=3$.
Denote $\mathcal{H}_{2,12}^{\prime(1)}=\{H^\prime|u\in U^\prime, vv^\prime,ww^\prime\in H^\prime,
uv, uw\not\in H^\prime, b=2, b_3=0\}$.
Since $n\geq 9$, without loss of generality, assume
$uu_0\in E(G)$, for every $H_1^\prime\in\mathcal{H}_{2,12}^{\prime(1)}$,
set $H_2^\prime=H_1^\prime-ww^\prime+uu^\prime-u_2w_2+uu_0$. It is obvious that
$H_2^\prime\in\mathcal{H}_{1,11}^{\prime(1)}$ in which $b=2, b_3=0$,
and denote this kind of subset of $\mathcal{H}_{1,11}^{\prime(1)}$
as $\mathcal{H}_{1,11}^{\prime(12)}$. Moreover,
$f_{11}: \mathcal{H}_{2,12}^{\prime(1)}\rightarrow \mathcal{H}_{1,11}^{\prime(12)}$ is an injection. Then
$$\sum_{H^\prime\in \mathcal{H}_{2,12}^{\prime(1)}\cup\mathcal{H}_{1,11}^{\prime(12)}} [W(H)-W(H^\prime)]\geq
\sum_{H^\prime\in \mathcal{H}_{2,12}^{\prime(1)}\cup f_{11}(\mathcal{H}_{2,12}^{\prime(1)})} (-6+6)\cdot N=0.$$

When $b=2, b_3=1$, then $g(G^\prime)=3$ and $|V(U^\prime)|=4$.
Denote $\mathcal{H}_{2,12}^{\prime(2)}=\{H^\prime|u\in U^\prime, vv^\prime,ww^\prime\in H^\prime,
uv, uw\not\in H^\prime, b=2, b_3=1\}$.
For every $H_1^\prime\in\mathcal{H}_{2,12}^{\prime(2)}$,
set $H_2^\prime=H_1^\prime-uw_2+uw$. It is obvious that
$H_2^\prime\in\mathcal{H}_{1,8}^{\prime(1)}$ in which $b_1=2, b_2=0, b_3=1$,
and denote this kind of subset of
$\mathcal{H}_{1,8}^{\prime(1)}$ as $\mathcal{H}_{1,8}^{\prime(12)}$.
Moreover,
$f_{12}: \mathcal{H}_{2,12}^{\prime(2)}\rightarrow \mathcal{H}_{1,8}^{\prime(12)}$ is an injection.
$$\sum_{H^\prime\in \mathcal{H}_{2,12}^{\prime(2)}\cup\mathcal{H}_{1,8}^{\prime(12)}} [W(H)-W(H^\prime)]\geq
\sum_{H^\prime\in \mathcal{H}_{2,12}^{\prime(2)}\cup f_{12}(\mathcal{H}_{2,12}^{\prime(2)})} (-1+3)\cdot N >0.$$

When $b=3, b_3=0$, then $g(G^\prime)=3$ and $|V(U^\prime)|=4$.
Denote $\mathcal{H}_{2,12}^{\prime(3)}=\{H^\prime|u\in U^\prime, vv^\prime,ww^\prime\in H^\prime,
uv, uw\not\in H^\prime, b=3, b_3=0\}$.
For every $H_1^\prime\in\mathcal{H}_{2,12}^{\prime(3)}$,
set $H_2^\prime=H_1^\prime-uw_2+uw$. It is obvious that
$H_2^\prime\in\mathcal{H}_{1,8}^{\prime(1)}$ in which $b_1=3, b_2=0, b_3=0$,
and denote this kind of subset of
$\mathcal{H}_{1,8}^{\prime(1)}$ as $\mathcal{H}_{1,8}^{\prime(13)}$.
Moreover, $f_{13}: \mathcal{H}_{2,12}^{\prime(3)}\rightarrow \mathcal{H}_{1,8}^{\prime(13)}$ is an injection.
$$\sum_{H^\prime\in \mathcal{H}_{2,12}^{\prime(3)}\cup\mathcal{H}_{1,8}^{\prime(13)}} [W(H)-W(H^\prime)]\geq
\sum_{H^\prime\in \mathcal{H}_{2,12}^{\prime(3)}\cup f_{13}(\mathcal{H}_{2,12}^{\prime(3)})} (-4+4)\cdot N=0.$$

When $b=4, b_3=0$, then $|V(U^\prime)|=5$ and $g(G^\prime)=3$ or $g(G^\prime)=5$.
Denote $\mathcal{H}_{2,12}^{\prime(4)}=\{H^\prime|u\in U^\prime, vv^\prime,ww^\prime\in H^\prime,
uv, uw\not\in H^\prime, b=4, b_3=0, g(G^\prime)=3\},
\mathcal{H}_{2,12}^{\prime(5)}=\{H^\prime|u\in U^\prime, vv^\prime,ww^\prime\in H^\prime,
uv, uw\not\in H^\prime, b=4, b_3=0, g(G^\prime)=5\}$.
For every $H_1^\prime\in\mathcal{H}_{2,12}^{\prime(4)}, H_2^\prime\in\mathcal{H}_{2,12}^{\prime(5)}$,
set $H_3^\prime=H_1^\prime-uw_2+uw, H_4^\prime=H_2^\prime-uw_2+uw$. It is obvious that
$H_3^\prime\in\mathcal{H}_{1,8}^{\prime(1)}$ in which $b_1=4, b_2=0, b_3=0, g(G^\prime)=3$,
and denote this kind of subset of
$\mathcal{H}_{1,8}^{\prime(1)}$ as $\mathcal{H}_{1,8}^{\prime(14)}$.
Moreover, $f_{14}: \mathcal{H}_{2,12}^{\prime(3)}\rightarrow \mathcal{H}_{1,8}^{\prime(14)}$ is an injection.
$$\sum_{H^\prime\in \mathcal{H}_{2,12}^{\prime(4)}\cup\mathcal{H}_{1,8}^{\prime(14)}} [W(H)-W(H^\prime)]\geq
\sum_{H^\prime\in \mathcal{H}_{2,12}^{\prime(4)}\cup f_{14}(\mathcal{H}_{2,12}^{\prime(4)})} (-2+6)\cdot N >0.$$
Furthermore,
$H_4^\prime\in\mathcal{H}_{1,8}^{\prime(1)}$ in which $b_1=4, b_2=0, b_3=0, g(G^\prime)=5$,
and denote this kind of subset of
$\mathcal{H}_{1,8}^{\prime(1)}$ as $\mathcal{H}_{1,8}^{\prime(15)}$.
Moreover, $f_{15}: \mathcal{H}_{2,12}^{\prime(5)}\rightarrow \mathcal{H}_{1,8}^{\prime(15)}$ is an injection.
$$\sum_{H^\prime\in \mathcal{H}_{2,12}^{\prime(5)}\cup\mathcal{H}_{1,8}^{\prime(15)}} [W(H)-W(H^\prime)]\geq
\sum_{H^\prime\in \mathcal{H}_{2,12}^{\prime(5)}\cup f_{15}(\mathcal{H}_{2,12}^{\prime(5)})} (-2+6)\cdot N >0.$$

Thus by summing over all possible subsets of $\mathcal{H}_i^\prime$,
from Theorem~\ref{theorem1.2} and $f$ is an injection on the whole.
Then $$\varphi_i(G^\prime)=\sum_{H^\prime\in\mathcal{H}_i^\prime} W(H^\prime)
<\sum_{H\in\mathcal{H}_i^*} W(H)\leq\sum_{H\in\mathcal{H}_i} W(H)=\varphi_i(G)$$ holds
for $i=2,3,\cdots,n-1$.

When $g$ is even, by Case 1, $\varphi_i(G)> \varphi_i(G^\prime), i=2,3,\cdots,n-1$ holds.
\end{proof}

\begin{remark}
Assume there are $t_i$ (resp. $t_{i+1}, t_{i+2}$) pendent edges and
$s_i$ (resp. $s_{i+1}, s_{i+2}$) pendent paths of length $2$ attached
to $u$ (resp. $v, w$).

If $t_i, t_{i+1}, t_{i+2} \neq 0$, without loss of generality,
assume $uu^{\prime\prime}, vv^{\prime\prime}, ww^{\prime\prime}\in M(G)$,
where $u^{\prime\prime}, v^{\prime\prime}, w^{\prime\prime}$ are pendent vertices
attached at $u,v,w$.
Then $M(G^\prime)=M(G)-\{uu^{\prime\prime}, vv^{\prime\prime}, ww^{\prime\prime}\}+\{uu^\prime, vv^\prime, ww^\prime\},$
we have $M(G^\prime)=M(G)$.
\end{remark}

\section{\large\bf{The ordering of graphs in the two sets $\mathcal{G}_3(s_1,t_1;s_2,t_2;s_3,t_3)$
and $\mathcal{G}_4(s_1,t_1;s_2,t_2;s_3,t_3;s_4,t_4)$  }}

\begin{lemma}\label{lemma3.1}
For an arbitrary graph $$G_4(s_1,t_1;s_2,t_2;s_3,t_3;s_4,t_4)
\in\mathcal{G}_4(s_1,t_1;s_2,t_2;s_3,t_3;s_4,t_4),$$
after removing all pendent edges or pendent paths of lengths $2$
at $u_2,u_3,u_4$, we obtain $G_4(\sum_{i=1}^4 s_i,\sum_{i=1}^4 t_i;0,0;0,0;0,0)$.
Then $$\varphi_i(G_4(s_1,t_1;s_2,t_2;s_3,t_3;s_4,t_4))\geq
\varphi_i(G_4(\sum_{i=1}^4 s_i,\sum_{i=1}^4 t_i;0,0;0,0;0,0)),$$
with equality if and only if $i\in\{0,1,n-1,n\}$. (see fig.5).
\end{lemma}

\begin{figure}
\begin{tikzpicture}

\node (v0)[draw,shape=circle,inner sep=1.5pt,label=45:$u_1$] at (-1,1){};
\node[draw,shape=circle,inner sep=1.5pt,label=0:$u_2$](v1) at (1,1){};
\node[draw,shape=circle,inner sep=1.5pt,label=0:$u_3$](v2) at (1,-1){};
\node[draw,shape=circle,inner sep=1.5pt,label=-45:$u_4$](v3) at (-1,-1){};
\node[draw,shape=circle,inner sep=1.5pt,label=0:](u2) at (130:3){};
\node[draw,shape=circle,inner sep=1.5pt,label=0:](u1) at (160:3){};
\node[label=0:$a$](a) at (145:3){};
\node[draw,shape=circle,inner sep=1.5pt,label=0:](u3) at (30:3){};
\node[draw,shape=circle,inner sep=1.5pt,label=0:](u4) at (60:3){};
\node[label=0:$b$](b) at (45:3){};
\node[draw,shape=circle,inner sep=1.5pt,label=0:](w2) at (-130:3){};
\node[draw,shape=circle,inner sep=1.5pt,label=0:](w1) at (-160:3){};
\node[label=0:$d$](d) at (-145:3){};
\node[draw,shape=circle,inner sep=1.5pt,label=0:](w3) at (-30:3){};
\node[draw,shape=circle,inner sep=1.5pt,label=0:](w4) at (-60:3){};
\node[label=0:$c$](c) at (-45:3){};

\draw (v0)--(v1) (v2)--(v1)(v3)--(v2)(v0)--(v3) (v0)--(u1) (v0)--(u2)
(v1)--(u3) (v1)--(u4) (v3)--(w1) (v3)--(w2)
(v2)--(w3) (v2)--(w4);
\draw [loosely dotted,thick] (u1) -- (u2) (u3) -- (u4)
(w1) -- (w2) (w3) -- (w4);
\node (v0)[draw,shape=circle,inner sep=1.5pt,label=45:$u_1$] at (5,1){};
\node[draw,shape=circle,inner sep=1.5pt,label=0:$u_2$](v1) at (7,1){};
\node[draw,shape=circle,inner sep=1.5pt,label=0:$u_3$](v2) at (7,-1){};

\node[draw,shape=circle,inner sep=1.5pt,label=-45:$u_4$](v3) at (5,-1){};
\node[draw,shape=circle,inner sep=1.5pt,label=0:](u2) at (4,2){};
\node[draw,shape=circle,inner sep=1.5pt,label=0:](u1) at (6,2){};
\node[label=-180:$a+b+c+d$](a) at (6,2.5){};

\draw (v0)--(v1) (v2)--(v1)(v3)--(v2)(v0)--(v3) (v0)--(u1) (v0)--(u2);
\draw [loosely dotted,thick] (u1) -- (u2);
\draw[->] (3.5,0)--(4.5,0);
\end{tikzpicture}
\caption{Transformation of Lemma 3.1} \label{fig:pepper}
\end{figure}
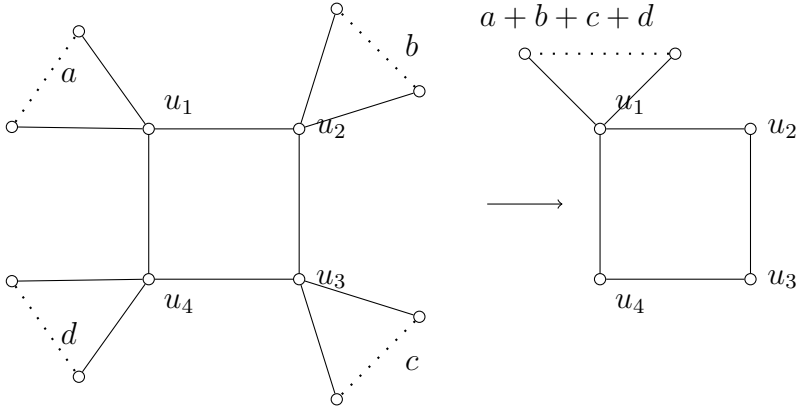

\begin{proof}
For convenience, we denote the graph $G_4(s_1,t_1;s_2,t_2;s_3,t_3;s_4,t_4)$
as $G$, and
$G_4(\sum_{i=1}^4 s_i,\sum_{i=1}^4 t_i;0,0;0,0;0,0)$ as $G^\prime$.

When $i\in\{0,1,n-1,n\}$, the proof is similar to Theorem~\ref{theorem2.2}.
For $2\leq i\leq n-2$,
we use $\mathcal{H}_i^\prime$ and $\mathcal{H}_i$ to denote
the set of all TU-subgraphs of $G$ and $G^\prime$ with
exactly $i$ edges, respectively. Let $\mathcal{H}_i^\prime=
\mathcal{H}_i^{\prime(1)}\cup\mathcal{H}_i^{\prime(2)}
\cup\mathcal{H}_i^{\prime(3)}\cup\mathcal{H}_i^{\prime(4)}$,
where $\mathcal{H}_i^{\prime(j)}(j=1,2,3,4)$ is the set of
TU-subgraphs in which $u_1,u_2,u_3,u_4$ belong to exactly
$j$ different components. Similarly, we can define $\mathcal{H}_i^{(j)}(j=1,2,3,4)$.

For an arbitrary TU-subgraph $H^\prime\in\mathcal{H}_i^{\prime(j)}$, let
$R^\prime$ be the component of $H^\prime$ containing $u_1$.
Let
$N_{R^\prime}(u_1)\cap N_{T^G_{u_1}}(u_1)=\{u_1^1,u_1^2,\cdots, u_1^a\}$,
where $0\leq b\leq \mbox{min}\{d_G(u_1)-2, |V(R^\prime)|-1\}$,
$N_{R^\prime}(u_1)\cap N_{T^G_{u_2}}(u_2)=\{u_2^1,u_2^2,\cdots, u_2^b\}$,
where $0\leq b\leq \mbox{min}\{d_G(u_2)-2, |V(R^\prime)|-1\}$,
$N_{R^\prime}(u_1)\cap N_{T^G_{u_3}}(u_3)=\{u_3^1,u_3^2,\cdots, u_3^c\}$,
where $0\leq c\leq \mbox{min}\{d_G(u_3)-2, |V(R^\prime)|-1\}$,
$N_{R^\prime}(u_1)\cap N_{T^G_{u_4}}(u_4)=\{u_4^1,u_4^2,\cdots, u_4^d\}$,
where $0\leq d\leq \mbox{min}\{d_G(u_4)-2, |V(R^\prime)|-1\}$.
For $G$, define $H$ with $V(H)=V(H^\prime)$ and
$$E(H)=E(H^\prime)-u_1u_2^1-\cdots-u_1u_2^b-u_1u_3^1-\cdots-u_1u_3^c-u_1u_4^1-\cdots-u_1u_4^d$$
$$+u_2u_2^1+\cdots+u_2u_2^b+u_3u_3^1+\cdots+u_3u_3^c+u_4u_4^1+\cdots+u_4u_4^d.$$
Then $H\in \mathcal{H}_i$. Obviously, $H^\prime\in\mathcal{H}_i^{\prime(j)}\Leftrightarrow
H\in \mathcal{H}_i^{(j)}, (j=1,2,3,4)$.
Let $f: \mathcal{H}_i^{\prime(j)}\rightarrow \mathcal{H}_i^{(j)}$, and $\mathcal{H}_i^{*(j)}=f(\mathcal{H}_i^{\prime(j)})=
\{f(H^\prime)|H^\prime\in\mathcal{H}_i^{\prime(j)}\}$.

Denote $|V(T^G_{u_1})\cap V(R^\prime)\setminus\{u_1\}|=A,
|V(T^G_{u_2})\cap V(R^\prime)\setminus \{u_2\}|=B,
|V(T^G_{u_3})\cap V(R^\prime)\setminus \{u_3\}|=C,
|V(T^G_{u_4})\cap V(R^\prime)\setminus \{u_4\}|=D$, where
$A,B,C,D\geq 0$.

Now we distinguish the proof into four cases.

Case 1: $H^\prime\in\mathcal{H}_i^{\prime(1)}$,
$u_1,u_2,u_3,u_4$ belong to one component, then $W(H)=W(H^\prime)$,
thus $$\sum_{H^\prime\in\mathcal{H}_i^{\prime(1)}}[W(H)-W(H^\prime)]=0.$$

Case 2: $H^\prime\in\mathcal{H}_i^{\prime(4)}$,
$u_1,u_2,u_3,u_4$ are in four trees, then $$W(H)-W(H^\prime)=
(A+1)(B+1)(C+1)(D+1) N_1-(A+B+C+D+1) N_1\geq0,$$
for some constant value $N_1$.
Thus $$\sum_{H^\prime\in\mathcal{H}_i^{\prime(4)}}[W(H)-W(H^\prime)]\geq0.$$

Case 3: $H^\prime\in\mathcal{H}_i^{\prime(2)}$,
$u_1,u_2,u_3,u_4$ are in two trees, then by computing,
$$W(H)-W(H^\prime)=[(A+D+2)(B+C+2)-2(A+B+C+D+2)] N_2$$$$+[(A+B+2)(C+D+2)-
2(A+B+C+D+2)] N_2$$$$+[AD+BD+CD+AC+BC+CD+AB+BD+BC+AB+AC+AD] N_2\geq0,$$
for some constant value $N_2$.
Thus $$\sum_{H^\prime\in\mathcal{H}_i^{\prime(2)}}[W(H)-W(H^\prime)]\geq0.$$

Case 4: $H^\prime\in\mathcal{H}_i^{\prime(3)}$,
$u_1,u_2,u_3,u_4$ are in three trees, then
$$W(H)-W(H^\prime)=[(AB+A+B)(C+D)+2AB+(AD+A+D)(B+C)+2AD$$$$+(CD+C+D)(A+B)+
2CD+(BC+B+C)(A+D)+2BC] N_3\geq0,$$
for some constant value $N_3$.
Thus $$\sum_{H^\prime\in\mathcal{H}_i^{\prime(3)}}[W(H)-W(H^\prime)]\geq0.$$

Now the inequality $\varphi_i(G)>\varphi_i(G^\prime), i=2,3,\cdots,n-2$ holds from
Theorem~\ref{theorem1.2} by summing over all possible subsets $\mathcal{H}_i^\prime$ of
TU-subgraphs $H^\prime$ of $G_1$ with $i$ edges.
\end{proof}

\begin{remark}\label{remark3.1}
If there is only one positive number in the set $\{t_1,t_2,t_3,t_4\}$,
then after the transformation in Lemma~\ref{lemma3.1}, $M(G)=M(G^\prime)$.

If $t_1=t_3=0, t_2,t_4\neq 0$ or $t_2=t_4=0, t_1,t_3\neq 0$,
then after the transformation in Lemma~\ref{lemma3.1}, $M(G)=M(G^\prime)$.
\end{remark}

Similar to the prove of Lemma~\ref{lemma3.1}, we have the the following Lemma.

\begin{lemma}\label{lemma3.2}
For an arbitrary graph $$G_3(s_1,t_1;s_2,t_2;s_3,t_3)
\in\mathcal{G}_3(s_1,t_1;s_2,t_2;s_3,t_3),$$
after removing all pendent edges or pendent paths of lengths $2$
at $u_2,u_3$, we obtain $G_3(\sum_{i=1}^3 s_i,\sum_{i=1}^3 t_i;0,0;0,0)$.
Then $$\varphi_i(G_3(s_1,t_1;s_2,t_2;s_3,t_3))\geq
\varphi_i(G_3(\sum_{i=1}^3 s_i,\sum_{i=1}^3 t_i;0,0;0,0)),$$
with equality if and only if $i\in\{0,1,n-1,n\}$.
\end{lemma}

\begin{remark}\label{remark3.2}
If there is at least one number in $\{t_1,t_2,t_3\}$ which equals to zero,
then $M(G_3(s_1,t_1;s_2,t_2;s_3,t_3))=M(G_3(\sum_{i=1}^3 s_i,\sum_{i=1}^3 t_i;0,0;0,0))$.
\end{remark}

\begin{lemma}\label{lemma3.3}
For a graph $$G_3(s_1,t_1;s_2,t_2;s_3,t_3)
\in\mathcal{G}_3(s_1,t_1;s_2,t_2;s_3,t_3),$$
which satisfies $t_1\neq 0, t_2\neq 0, t_3\neq 0$,
without loss of generality,
assume $u_1u^\prime, u_2v^\prime, u_3w^\prime$
are pendent edges of $G_3(s_1,t_1;s_2,t_2;s_3,t_3)$.
After removing all pendent edges or pendent paths of
lengths $2$ except $u_1u^\prime, u_2v^\prime$
from vertices $u_1,u_2$ to vertex $u_3$, we obtain
$G_3(0,1;0,1;\sum_{i=1}^3 s_i,\sum_{i=1}^3 t_i-2)$.
Then $$\varphi_i(G_3(s_1,t_1;s_2,t_2;s_3,t_3))\geq
\varphi_i(G_3(0,1;0,1;\sum_{i=1}^3 s_i,\sum_{i=1}^3 t_i-2)),$$
with equality if and only if $i\in\{0,1,n-1,n\}$.
\end{lemma}

\begin{proof}
For convenience, we denote the graph $G_3(s_1,t_1;s_2,t_2;s_3,t_3)$
as $G$, and denote
$G_3(0,1;0,1;\sum_{i=1}^3 s_i,\sum_{i=1}^3 t_i-2)$ as $G^\prime$.

When $i\in\{0,1,n-1,n\}$, the proof is similar to Theorem~\ref{theorem2.2}.
For $2\leq i\leq n-2$,
we use $\mathcal{H}_i^\prime$ and $\mathcal{H}_i$ to denote
the set of all TU-subgraphs of $G$ and $G^\prime$ with
exactly $i$ edges, respectively. Let $\mathcal{H}_i^\prime=
\mathcal{H}_i^{\prime(1)}\cup\mathcal{H}_i^{\prime(2)}
\cup\mathcal{H}_i^{\prime(3)}$,
where $\mathcal{H}_i^{\prime(j)}(j=1,2,3)$ is the set of
TU-subgraphs in which $u_1,u_2,u_3$ belong to exactly
$j$ different components. Similarly, we can define $\mathcal{H}_i^{(j)}(j=1,2,3)$.

For an arbitrary TU-subgraph $H^\prime\in\mathcal{H}_i^{\prime(j)}$, let
$R^\prime$ be the component of $H^\prime$ containing $u_3$.
Let
$N_{R^\prime}(u_3)\cap N_{T^G_{u_1}}(u_1)=\{u_1,u_1^1,u_1^2,\cdots, u_1^a\}$,
where $0\leq b\leq \mbox{min}\{d_G(u_1)-2, |V(R^\prime)|-1\}$,
$N_{R^\prime}(u_3)\cap N_{T^G_{u_2}}(u_2)=\{u_2,u_2^1,u_2^2,\cdots, u_2^b\}$,
where $0\leq b\leq \mbox{min}\{d_G(u_2)-2, |V(R^\prime)|-1\}$,
$N_{R^\prime}(u_3)\cap N_{T^G_{u_3}}(u_3)=\{u_3^1,u_3^2,\cdots, u_3^c\}$,
where $0\leq c\leq \mbox{min}\{d_G(u_3)-2, |V(R^\prime)|-1\}$.
For $G$, define $H$ with $V(H)=V(H^\prime)$ and
$$E(H)=E(H^\prime)-u_3u_1^1-\cdots-u_3u_1^a-u_3u_2^1-\cdots-u_3u_2^b$$
$$+u_1u_1^1+\cdots+u_1u_1^a+u_2u_2^1+\cdots+u_2u_2^b.$$
Then $H\in \mathcal{H}_i$. Obviously, $H^\prime\in\mathcal{H}_i^{\prime(j)}\Leftrightarrow
H\in \mathcal{H}_i^{(j)}, (j=1,2,3)$.
Let $f: \mathcal{H}_i^{\prime(j)}\rightarrow \mathcal{H}_i^{(j)}$, and $\mathcal{H}_i^{*(j)}=f(\mathcal{H}_i^{\prime(j)})=
\{f(H^\prime)|H^\prime\in\mathcal{H}_i^{\prime(j)}\}$.

Denote $|V(T^G_{u_1})\cap V(R^\prime)\setminus\{u_1,u^\prime\}|=A,
|V(T^G_{u_2})\cap V(R^\prime)\setminus \{u_2,v^\prime\}|=B,
|V(T^G_{u_3})\cap V(R^\prime)\setminus \{u_3,w^\prime\}|=C$, where
$A,B,C\geq 0$. Denote $N$ be the product of the orders of
all components of $H^\prime$ excluding $u_1,u^\prime,u_2,v^\prime,u_3,w^\prime$.

Now we distinguish the proof into three cases.

Case 1: $H^\prime\in\mathcal{H}_i^{\prime(1)}$,
$u_1,u_2,u_3$ belong to one component, then $W(H)=W(H^\prime)$,
thus $$\sum_{H^\prime\in\mathcal{H}_i^{\prime(1)}}[W(H)-W(H^\prime)]=0.$$

Case 2: $H^\prime\in\mathcal{H}_i^{\prime(2)}$,
$u_1,u_2,u_3$ are in two trees, we distinguish this
case into the following four cases.

Subcase 2.1: $|\{u_1u^\prime, u_2v^\prime, u_3w^\prime\}\cap E(H^\prime)|=1$, and $N$ is fixed.

Denote this kind of subset of $\mathcal{H}_i^{\prime(2)}$
as $\mathcal{H}_i^{\prime(21)}$.

(1). Assume $u_1u^\prime\in E(H^\prime)$.

(1.1). $u_1u_3\in E(H^\prime), u_1u_2,u_2u_3\not\in E(H^\prime)$.
$$W(H)-W(H^\prime)=
(A+C+3)(B+1) N-(A+B+C+3) N.$$

(1.2). $u_1u_2\in E(H^\prime), u_1u_3,u_2u_3\not\in E(H^\prime)$.
$$W(H)-W(H^\prime)=
(A+B+3)(C+1) N-3(A+B+C+1) N.$$

(1.3). $u_2u_3\in E(H^\prime), u_1u_2,u_1u_3\not\in E(H^\prime)$.
$$W(H)-W(H^\prime)=
(B+C+3)(A+1) N-2(A+B+C+2) N.$$

(2). Assume $u_2v^\prime\in E(H^\prime)$. By the symmetry of $u_1$ and $u_2$,
the discussion is similar to the above.

(3). Assume $u_3w^\prime\in E(H^\prime)$.

(3.1). $u_1u_3\in E(H^\prime), u_1u_2,u_2u_3\not\in E(H^\prime)$.
$$W(H)-W(H^\prime)=
(A+C+3)(B+1) N-(A+B+C+3) N.$$

(3.2). $u_1u_2\in E(H^\prime), u_1u_3,u_2u_3\not\in E(H^\prime)$.
$$W(H)-W(H^\prime)=
(A+B+2)(C+2) N-2(A+B+C+2) N.$$

(3.3). $u_2u_3\in E(H^\prime), u_1u_2,u_1u_3\not\in E(H^\prime)$.
$$W(H)-W(H^\prime)=(B+C+3)(A+1) N-(A+B+C+3) N.$$

Then
$$\sum_{H^\prime\in\mathcal{H}_i^{\prime(21)}}[W(H)-W(H^\prime)]=6(AB+BC+AC)N\geq 0.$$

Subcase 2.2: $|\{u_1u^\prime, u_2v^\prime, u_3w^\prime\}\cap E(H^\prime)|=2$, and $N$ is fixed.

Denote this kind of subset of $\mathcal{H}_i^{\prime(2)}$
as $\mathcal{H}_i^{\prime(22)}$.

(1). Assume $u_1u^\prime,u_2v^\prime\in E(H^\prime)$.

(1.1). $u_1u_3\in E(H^\prime), u_1u_2,u_2u_3\not\in E(H^\prime)$.
$$W(H)-W(H^\prime)=
(A+C+3)(B+2) N-2(A+B+C+3) N.$$

(1.2). $u_1u_2\in E(H^\prime), u_1u_3,u_2u_3\not\in E(H^\prime)$.
$$W(H)-W(H^\prime)=
(A+B+4)(C+1) N-4(A+B+C+1) N.$$

(1.3). $u_2u_3\in E(H^\prime), u_1u_2,u_1u_3\not\in E(H^\prime)$.
$$W(H)-W(H^\prime)=
(B+C+3)(A+2) N-2(A+B+C+3) N.$$

(2). Assume $u_1u^\prime,u_3w^\prime\in E(H^\prime)$.

(2.1). $u_1u_3\in E(H^\prime), u_1u_2,u_2u_3\not\in E(H^\prime)$.
$$W(H)-W(H^\prime)=
(A+C+4)(B+1) N-(A+B+C+4) N.$$

(2.2). $u_1u_2\in E(H^\prime), u_1u_3,u_2u_3\not\in E(H^\prime)$.
$$W(H)-W(H^\prime)=
(A+B+3)(C+2) N-3(A+B+C+2) N.$$

(2.3). $u_2u_3\in E(H^\prime), u_1u_2,u_1u_3\not\in E(H^\prime)$.
$$W(H)-W(H^\prime)=(B+C+3)(A+2) N-2(A+B+C+3) N.$$

(3). Assume $u_2v^\prime,u_3w^\prime\in E(H^\prime)$. By the symmetry of $A$ and $B$,
the discussion is similar to the above.

Then
$$\sum_{H^\prime\in\mathcal{H}_i^{\prime(22)}}[W(H)-W(H^\prime)]=6(AB+BC+AC)N\geq 0.$$

Subcase 2.3: $|\{u_1u^\prime, u_2v^\prime, u_3w^\prime\}\cap E(H^\prime)|=3$, and $N$ is fixed.

Denote this kind of subset of $\mathcal{H}_i^{\prime(2)}$
as $\mathcal{H}_i^{\prime(23)}$.

(1). $u_1u_3\in E(H^\prime), u_1u_2,u_2u_3\not\in E(H^\prime)$.
$$W(H)-W(H^\prime)=
(A+C+4)(B+2) N-2(A+B+C+4) N.$$

(2). $u_1u_2\in E(H^\prime), u_1u_3,u_2u_3\not\in E(H^\prime)$.
$$W(H)-W(H^\prime)=
(A+B+4)(C+2) N-4(A+B+C+2) N.$$

(3). $u_2u_3\in E(H^\prime), u_1u_2,u_1u_3\not\in E(H^\prime)$.
$$W(H)-W(H^\prime)=
(B+C+4)(A+2) N-2(A+B+C+4) N.$$

Then
$$\sum_{H^\prime\in\mathcal{H}_i^{\prime(23)}}[W(H)-W(H^\prime)]=2(AB+BC+AC)N\geq 0.$$

Subcase 2.4: $|\{u_1u^\prime, u_2v^\prime, u_3w^\prime\}\cap E(H^\prime)|=0$, and $N$ is fixed.

Denote this kind of subset of $\mathcal{H}_i^{\prime(2)}$
as $\mathcal{H}_i^{\prime(23)}.$

(1). $u_1u_3\in E(H^\prime), u_1u_2,u_2u_3\not\in E(H^\prime)$.
$$W(H)-W(H^\prime)=
(A+C+2)(B+1) N-(A+B+C+2) N.$$

(2). $u_1u_2\in E(H^\prime), u_1u_3,u_2u_3\not\in E(H^\prime)$.
$$W(H)-W(H^\prime)=
(A+B+2)(C+1) N-2(A+B+C+1) N.$$

(3). $u_2u_3\in E(H^\prime), u_1u_2,u_1u_3\not\in E(H^\prime)$.
$$W(H)-W(H^\prime)=
(B+C+2)(A+1) N-(A+B+C+2) N.$$

Then
$$\sum_{H^\prime\in\mathcal{H}_i^{\prime(24)}}[W(H)-W(H^\prime)]=2(AB+BC+AC)N\geq 0.$$

After summing all above results in Case 2 we have
$$\sum_{H^\prime\in\mathcal{H}_i^{\prime(2)}}[W(H)-W(H^\prime)]
=\sum_{H^\prime\in\mathcal{H}_i^{\prime(21)}\cup\mathcal{H}_i^{\prime(22)}
\cup\mathcal{H}_i^{\prime(23)}\cup\mathcal{H}_i^{\prime(24)}}[W(H)-W(H^\prime)]\geq 0.$$

Case 3: $H^\prime\in\mathcal{H}_i^{\prime(3)}$,
$u_1,u_2,u_3$ are in three trees, we again distinguish the following four cases.

Subcase 3.1: $|\{u_1u^\prime, u_2v^\prime, u_3w^\prime\}\cap E(H^\prime)|=1$, and $N$ is fixed.

Denote this kind of subset of $\mathcal{H}_i^{\prime(3)}$
as $\mathcal{H}_i^{\prime(31)}$.

(1).$u_1u^\prime\in E(H^\prime)$.

$$W(H)-W(H^\prime)=
(A+2)(B+1)(C+1) N-2(A+B+C+1) N.$$

(2).$u_2v^\prime\in E(H^\prime)$.

$$W(H)-W(H^\prime)=
(A+1)(B+2)(C+1) N-2(A+B+C+1) N.$$

(3).$u_3w^\prime\in E(H^\prime)$.

$$W(H)-W(H^\prime)=
(A+1)(B+1)(C+2) N-(A+B+C+2) N.$$

Then
$$\sum_{H^\prime\in\mathcal{H}_i^{\prime(31)}}[W(H)-W(H^\prime)]=(3ABC+4AB+4BC+4AC)N\geq 0.$$

Subcase 3.2: $|\{u_1u^\prime, u_2v^\prime, u_3w^\prime\}\cap E(H^\prime)|=2$, and $N$ is fixed.

Denote this kind of subset of $\mathcal{H}_i^{\prime(3)}$
as $\mathcal{H}_i^{\prime(32)}$.

(1).$u_1u^\prime,u_3w^\prime\in E(H^\prime)$.

$$W(H)-W(H^\prime)=
(A+2)(B+1)(C+2) N-2(A+B+C+2) N.$$

(2).$u_2v^\prime,u_3w^\prime\in E(H^\prime)$.

$$W(H)-W(H^\prime)=
(A+1)(B+2)(C+2) N-2(A+B+C+2) N.$$

(3).$u_1u^\prime,u_2v^\prime\in E(H^\prime)$.

$$W(H)-W(H^\prime)=
(A+2)(B+2)(C+1) N-4(A+B+C+1) N.$$

Then
$$\sum_{H^\prime\in\mathcal{H}_i^{\prime(32)}}[W(H)-W(H^\prime)]=(3ABC+5AB+5BC+5AC)N\geq 0.$$

Subcase 3.3: $|\{u_1u^\prime, u_2v^\prime, u_3w^\prime\}\cap E(H^\prime)|=3$, and $N$ is fixed.

Denote this kind of subset of $\mathcal{H}_i^{\prime(3)}$
as $\mathcal{H}_i^{\prime(33)}$.

$$W(H)-W(H^\prime)=
(A+2)(B+2)(C+2) N-4(A+B+C+2) N.$$

Then
$$\sum_{H^\prime\in\mathcal{H}_i^{\prime(33)}}[W(H)-W(H^\prime)]=(ABC+2AB+2BC+2AC)N\geq 0.$$

Subcase 3.4: $|\{u_1u^\prime, u_2v^\prime, u_3w^\prime\}\cap E(H^\prime)|=0$, and $N$ is fixed.

Denote this kind of subset of $\mathcal{H}_i^{\prime(3)}$
as $\mathcal{H}_i^{\prime(34)}$.

$$W(H)-W(H^\prime)=
(A+1)(B+1)(C+1) N-(A+B+C+1) N.$$

Then
$$\sum_{H^\prime\in\mathcal{H}_i^{\prime(34)}}[W(H)-W(H^\prime)]=(ABC+AB+BC+AC)N\geq 0.$$

After summing all above results in Case 3 we have
$$\sum_{H^\prime\in\mathcal{H}_i^{\prime(3)}}[W(H)-W(H^\prime)]
=\sum_{H^\prime\in\mathcal{H}_i^{\prime(31)}\cup\mathcal{H}_i^{\prime(32)}
\cup\mathcal{H}_i^{\prime(33)}\cup\mathcal{H}_i^{\prime(34)}}[W(H)-W(H^\prime)]\geq 0.$$

By the assumption of this Lemma, there is at least
one TU-subgraph $H^\prime$ with $AB\geq 1$ or $BC\geq 1$ or $AC\geq 1$,
thus we can get
$$\varphi_i(G)-\varphi_i(G^\prime)=\sum_{H\in\mathcal{H}_i}W(H)-\sum_{H^\prime\in\mathcal{H}_i^{\prime}}W(H^\prime)
=\sum_{j=1}^3(\sum_{H\in\mathcal{H}_i^{(j)}}W(H)-\sum_{H^\prime\in\mathcal{H}_i^{\prime(j)}}W(H^\prime))>0.$$
\end{proof}

\begin{remark}\label{remark3.3}
If $t_1\neq 0, t_2\neq 0, t_3\neq 0$,
then
$$M(G_3(s_1,t_1;s_2,t_2;s_3,t_3))=M(G_3(0,1;0,1;\sum_{i=1}^3 s_i,\sum_{i=1}^3 t_i-2)).$$
\end{remark}

Similar to the prove of Lemma~\ref{lemma3.3}, we have the the following Lemma.

\begin{lemma}\label{lemma3.4}

(1). For a graph $$G_4(s_1,t_1;s_2,t_2;s_3,t_3;s_4,t_4)
\in\mathcal{G}_4(s_1,t_1;s_2,t_2;s_3,t_3;s_4,t_4),$$
which satisfies $t_1\geq 1, t_2\geq 1, t_3\geq 1, t_4\geq 1$,
without loss of generality,
assume $u_1u_1^\prime, u_2u_2^\prime, u_3u_3^\prime,u_4u_4^\prime$
are pendent edges of $G_4(s_1,t_1;s_2,t_2;s_3,t_3;s_4,t_4)$.
After removing all pendent edges or pendent paths of
lengths $2$ except $u_1u_1^\prime, u_2u_2^\prime, u_3u_3^\prime$
from vertices $u_1,u_2,u_3$ to vertex $u_4$, we obtain
$G_4(0,1;0,1;0,1;\sum_{i=1}^4 s_i,\sum_{i=1}^4 t_i-3)$.
Then
$$\varphi_i(G_4(s_1,t_1;s_2,t_2;s_3,t_3;s_4,t_4))\geq
\varphi_i(G_4(0,1;0,1;0,1;\sum_{i=1}^4 s_i,\sum_{i=1}^4 t_i-3)).$$

(2). For a graph
$$G_4(s_1,t_1;s_2,t_2;s_3,t_3;s_4,t_4)
\in\mathcal{G}_4(s_1,t_1;s_2,t_2;s_3,t_3;s_4,t_4),$$
which satisfies $t_1=0, t_2\geq 1, t_3\geq 1, t_4\geq 1$,
without loss of generality,
assume $u_2u_2^\prime, u_3u_3^\prime,u_4u_4^\prime$
are pendent edges of $G_4(s_1,t_1;s_2,t_2;s_3,t_3;s_4,t_4)$.
After removing all pendent edges or pendent paths of
lengths $2$ except $u_2u_2^\prime, u_4u_4^\prime$
from vertices $u_2,u_4$ to vertex $u_3$, we obtain
$G_4(0,1;0,1;\sum_{i=1}^4 s_i,\sum_{i=1}^4 t_i-2;0,1)$.
Then
$$\varphi_i(G_4(s_1,t_1;s_2,t_2;s_3,t_3;s_4,t_4))\geq
\varphi_i(G_4(0,1;0,1;\sum_{i=1}^4 s_i,\sum_{i=1}^4 t_i-2;0,1)).$$

(3). For a graph
$$G_4(s_1,t_1;s_2,t_2;s_3,t_3;s_4,t_4)
\in\mathcal{G}_4(s_1,t_1;s_2,t_2;s_3,t_3;s_4,t_4),$$
which satisfies $t_1=0, t_2=0, t_3\geq 1, t_4\geq 1$,
without loss of generality,
assume $u_3u_3^\prime,u_4u_4^\prime$
are pendent edges of $G_4(s_1,t_1;s_2,t_2;s_3,t_3;s_4,t_4)$.
After removing all pendent edges or pendent paths of
lengths $2$ except $u_3u_3^\prime$
from vertices $u_3$ to vertex $u_4$, we obtain
$G_4(0,0;0,0;0,1;\sum_{i=1}^4 s_i,\sum_{i=1}^4 t_i-1)$.
Then
$$\varphi_i(G_4(s_1,t_1;s_2,t_2;s_3,t_3;s_4,t_4))\geq
\varphi_i(G_4(0,0;0,0;0,1;\sum_{i=1}^4 s_i,\sum_{i=1}^4 t_i-1)).$$
\end{lemma}

\begin{remark}\label{remark3.4}
If $t_1, t_2, t_3, t_4\geq 1$,
then $$M(G_4(s_1,t_1;s_2,t_2;s_3,t_3;s_4,t_4))=
M(G_4(0,1;0,1;0,1;\sum_{i=1}^4 s_i,\sum_{i=1}^4 t_i-3)).$$

If $t_1=0, t_2\geq 1, t_3\geq 1, t_4\geq 1$,
then $$M(G_4(s_1,t_1;s_2,t_2;s_3,t_3;s_4,t_4))=
M(G_4(0,1;0,1;\sum_{i=1}^4 s_i,\sum_{i=1}^4 t_i-2;0,1)).$$

If $t_1=t_2=0, t_3\geq 1, t_4\geq 1$,
then $$M(G_4(s_1,t_1;s_2,t_2;s_3,t_3;s_4,t_4))=
M(G_4(0,0;0,0;0,1;\sum_{i=1}^4 s_i,\sum_{i=1}^4 t_i-1)).$$

\end{remark}

For any graph $G$ and $v\in V(G)$, let $Q_{G|v}(x)$
be the principal submatrix of $Q_G(x)$ obtained by deleting the row
and column corresponding to the vertex $v$. Similar to the proof of
Theorem~\ref{theorem1.5} and Theorem~\ref{theorem1.6}, we can prove the
following three lemmas.

\begin{lemma}\label{lemma3.5}
If $G=G_1|u : G_2|v$, then
$Q_G(x)=Q_{G_1}(x)Q_{G_2}(x)-Q_{G_1}(x)Q_{G_2|v}(x)-Q_{G_2}(x)Q_{G_1|u}(x)$.
\end{lemma}

\begin{lemma}\label{lemma3.6}
If $G$ be a connected graph with $n$ vertices which consists of a subgraph $H (V(H)\geq 2)$
and $n-|V(H)|$ pendent vertices attached to a vertex $v$ in $H$, then
$Q_G(x)=(x-1)^{(n-|V(H)|)}Q_{H}(x)-(n-|V(H)|)x(x-1)^{(n-|V(H)|-1)}Q_{H|v}(x)$.
\end{lemma}

\begin{lemma}\label{lemma3.7}
If $G$ be a connected graph with $n$ vertices which consists of a subgraph $H (V(H)\geq 2)$
and $\tfrac{n-|V(H)|}{2}$ pendent paths of length $2$ attached to a vertex $v$ in $H$, denote
$\tfrac{n-|V(H)|}{2}=k$, then
$Q_G(x)=(x^2-3x+1)^k Q_{H}(x)-kx(x-2)(x^2-3x+1)^{k-1}Q_{H|v}(x)$.
\end{lemma}

Let $f(x)=\sum_{i=0}^n (-1)^ia_ix^{n-i}, g(x)=\sum_{j=0}^m (-1)^ja_jx^{m-j}, a_i>0, b_j>0$.
Then it is easy to see $f(x)g(x)=f(x)=\sum_{k=0}^{m+n} (-1)^k \sum_{i=0}^k a_ib_{k-i}x^{n+m-k}$
has coefficients alternate with positive and negative.

Denote $$G_3^1=G_3(0,0;0,0;m-2,n-2m+1),$$ $$G_3^2=G_3(0,1;0,1;m-3,n-2m+1),$$
and $$G_4^1=G_4(0,0;0,0;0,0;m-2,n-2m),$$ $$G_4^2=G_4(0,0;0,0;0,1;m-3,n-2m+1),$$
$$G_4^3=G_4(0,0; 0,1; m-3,n-2m; 0,1),$$ $$G_4^4=G_4(0,1;0,1;0,1;m-4,n-2m+1).$$

By using Lemma~\ref{lemma3.6} and Lemma~\ref{lemma3.7}, we can
compute the signless Laplacian polynomials of several special
$n$-vertex unicyclic graphs with fixed matching number $m$.
For convenience, write $Q_G(x)$ as $Q(G, x)$.

$Q(G_3^1,x)=(x^2-3x+1)^{(m-3)}(x-1)^{(n-2m+1)}[x^5+(-8-n+m)x^4+(-6m+22+6n)x^3+(9m-25-10n)x^2+(12+3n)x-4],$

$Q(G_3^2,x)=(x^2-3x+1)^{(m-3)}(x-1)^{(n-2m)}[x^6+(-9-n+m)x^5+(-8m+27+8n)x^4+(18m-36-19n)x^3+(24+14n-9m)x^2+(-12-3n)x+4],$

$Q(G_4^1,x)=(x^2-3x+1)^{(m-3)}(x-1)^{(n-2m-1)}x(x-2)[x^5+(-8-n+m)x^4+(-7m+22+7n)x^3+(14m-25-15n)x^2+(10+10n-6m)x-2n],$

$Q(G_4^2,x)=(x^2-3x+1)^{(m-3)}(x-1)^{(n-2m+1)}x[x^4+(-8-n+m)x^3+(-7m+20+7n)x^2+(12m-17-13n)x+4n],$

$Q(G_4^3,x)=(x^2-3x+1)^{(m-4)}(x-1)^{(n-2m-1)}x(x^2-4x+2)[x^6+(-9-n+m)x^5+(27+9n-9m)x^4+(26m-33-27n)x^3+(32n-26m+14)x^2+(-14n+6m)x+2n],$

$Q(G_4^4,x)=(x^2-3x+1)^{(m-5)}(x-1)^{(n-2m)}x(x^2-4x+2)
[x^7+(-11-n+m)x^6+(11n-11m+43)x^5+(-43n+42m-74)x^4+(50+74n-66m)x^3+(-56n+38m)x^2+(-8+18n-6m)x-2n]$.

Then we have
\begin{eqnarray}
&Q(G_3^1,x)-Q(G_3^2,x)\nonumber\\
&=(x^2-3x+1)^{(m-3)}(x-1)^{(n-2m)}F_1(x)
\end{eqnarray}
where $F_1(x)=-(n-m-3)x^4+(3n-3m-11)x^3-(n-13)x^2-4x$.

If $n-m\leq 2$ and $n\leq 12$,
it is obvious that $\mbox{Eq.}(1)$ is a polynomial on $x$ with order $n-2$,
and each factor of $\mbox{Eq.}(1)$ is a real polynomial with alternate
coefficients on positive and negative, assume $\mbox{Eq.}(1)=
\sum_{i=2}^{n-1} (-1)^i a_ix^{n-i}$, where $a_i> 0$ for
$i=2,3,\cdots,n-1$. Notice that $a_i=\varphi_i(G_3^1)-\varphi_i(G_3^2)>0$,
thus $\varphi_i(G_3^1)>\varphi_i(G_3^2)$ in this case.
If $n=6, m=3$, then $\mbox{Eq.}(1)=-2x^3+7x^2-4x=
\sum_{i=3}^{5} (-1)^i a_ix^{6-i}$, where $a_i> 0$.
Since $a_i=\varphi_i(G_3^1)-\varphi_i(G_3^2)$,
thus $\varphi_i(G_3^1)>\varphi_i(G_3^2)$ in this case.
For other case, it is easy to find that all the signless
Laplacian coefficients of $G_3^1$ and $G_3^2$ are not comparable.

And,
\begin{eqnarray}
&Q(G_4^1,x)-Q(G_4^2,x)\nonumber\\
&=(x^2-3x+1)^{(m-3)}(x-1)^{(n-2m-1)}x F_2(x)
\end{eqnarray}
where $F_2(x)=x^4-(n-m+4)x^3+(3n-3m+6)x^2-(n+3)x$.

By the fact $n>m$, in a similar way to the above discussion,
it is easy to see $\varphi_i(G_4^1)\geq\varphi_i(G_4^2)$, and
for $i=2,3,\cdots,n-1$, the inequality is strict.

And,
\begin{eqnarray}
&Q(G_4^3,x)-Q(G_4^4,x)\nonumber\\
&=(x^2-3x+1)^{(m-5)}(x-1)^{(n-2m-1)}x(x^2-4x+2) F_3(x)
\end{eqnarray}
where $F_3(x)=x^6-(n-m+6)x^5+(5n-5m+16)x^4-(7n-6m+25)x^3+(2n+22)x^2-8x$.

Similarly,  $\varphi_i(G_4^3)\geq\varphi_i(G_4^4)$ holds, and
for $i=2,3,\cdots,n-1$, the inequality is strict.

Furthermore,
\begin{eqnarray}
&Q(G_4^4,x)-Q(G_4^2,x)\nonumber\\
&=(x^2-3x+1)^{(m-5)}(x-1)^{(n-2m)}x F_4(x)
\end{eqnarray}
where $F_4(x)=(n-m-4)x^7-(10n-10m-42)x^6+(37n-36m-170)x^5-(63n-56m-338)x^4+(51n-36m-345)x^3-(20n-9m-171)x^2+(3n-33)x$.

When $n=8$, by using Matlab 7.0, $\varphi_i(G_4^2)\geq\varphi_i(G_4^4)$ holds.
When $n=9,10$, by using Matlab 7.0, we can find
that all the signless Laplacian coefficients of
$G_4^2$ and $G_4^4$ are not comparable.
If $n\geq 11$, and at least one of the following statements holds:

(1). $m\geq 4, n-m\geq 7$.

(2). $m\geq 5, n-m\geq 6$.

(3). $m\geq 7, n-m\geq 5$.\\
In a similar way to the above discussion,
it is easy to see $\varphi_i(G_4^4)\geq\varphi_i(G_4^2)$, and
for $i=2,3,\cdots,n-1$, the inequality is strict.
For other case, we claim that all the signless
Laplacian coefficients of $G_4^2$ and $G_4^4$ are not comparable.

\begin{remark}\label{remark3.8}
For a fixed value $n$, and for an arbitrary graph $H_1\in\mathcal{G}_3(n)$ 
and an arbitrary graph $H_2\in\mathcal{G}_4(n)$, all the signless Laplacian
coefficients of $H_1$ and $H_2$ are not comparable. Since $\varphi_n(H_1)
=4, \varphi_n(H_2)=0$, but $\varphi_{n-1}(H_1)=3n, \varphi_{n-1}(H_2)=4n$.

For fixed values $n$ and $m$, and for arbitrary graphs 
$U_1\in\mathcal{G}_3(n,m), U_2\in\mathcal{G}_4(n,m)$, it 
is easy to see that $\varphi_{n-1}(U_1)<\varphi_{n-1}(U_2), 
\varphi_{n}(U_1)>\varphi_{n}(U_2)$, 
then $H_1$ and $H_2$ are not comparable with regard 
to all the signless Laplacian coefficients.
\end{remark}

\section{\large\bf{The Signless Laplacian coefficients of unicyclic graphs in $\mathcal{G}(n,m)$}}

\begin{theorem}\label{thoerem4.1}
In the set of all n-vertex unicyclic graphs in $\mathcal{G}_{g_1}(n,m)$, then

(1). If $m=2$, then $G_3^1$ has minimal 
all the signless Laplacian coefficients. 

(2). If $n-m\leq 3$ and $6 \leq n\leq 12, m\geq 3$, 
$G_3^2$ has minimal all the signless Laplacian coefficients 
$\varphi_i, i=0,1,2,\cdots,n$.

(3). If $n-m> 3, m\geq 3$ or $n> 12, m\geq 3$, $G_3^1$ and $G_3^2$ 
are two extremal graphs which have minimal signless Laplacian coefficients
$\varphi_i, i=0,1,2,\cdots,n$ in $\mathcal{G}_{g_1}(n,m)$. Furthermore, 
$G_3^1$ and $G_3^2$ can not be compared.
\end{theorem}

\begin{proof}
For $n=5$,
the matching numbers of all unicyclic graphs in $\mathcal{G}_{g_1}(n,m)$ are $2$,
so by using Matlab 7.0, we have 
$G_3^1$ has minimum all the signless Laplacian 
coefficients in $\mathcal{G}_{g_1}(5,m)$. Next assume $n\geq6$.

For fixed values $n, m$, let $G$ be an arbitrary graph in $\mathcal{G}_{g_1}(n,m)$.
Let $M(G)$ denote a maximum matching of $G$ containing the most pendent edges.
Now we need to prove after series of transformations, $G$ will
become to $U^\prime$ and $U^\prime$ has minimum signless Laplacian coefficients
in $\mathcal{G}_{g_1}(n,m)$.

Step 1: When there is a pendent path $u_1u_2\cdots u_k$ of length at least $3$
in $G$, where $d(u_1)\geq 3, d(u_2)=d(u_3)=\cdots=d(u_{k-1})=2, d(u_k)=1$
and $k\geq 3$. Assume $u_{k-3}u_{k-2}\in M(G)$, since
$E_{u_{k-3}u_{k-2}}^{u_{k-2}}\cap M(G)=\emptyset$, after performing the transformation
of Definition~\ref{definition2.1} to $u_{k-3}u_{k-2}$, we have $M(G)=M(G_{u_{k-3}u_{k-2}})$,
thus $G_{u_{k-3}u_{k-2}}\in\mathcal{G}_{g_1}(n,m)$, and
$\varphi_i(G)>\varphi_i(G_{u_{k-3}u_{k-2}})$, $i=2,3,\cdots,n-1$ by Theorem~\ref{theorem2.2}.

After performing transformations in Step 1 consecutively, we
have that the graph in which each pendent path
has length at most $2$ has smaller signless Laplacian
coefficients in $\mathcal{G}_{g_1}(n,m)$.

Step 2: For $u\in \{u: dist(u,C)\geq dist(u^\prime,C), d(u), d(u^\prime)\geq 3\}$,
$v$ is a neighbor of $u$, which satisfies $dist(u,C)-1=dist(v,C)$.

Case 2.1: When there is a pendent
edge $uu^\prime$ at $u$, then after performing the transformation
of Definition~\ref{definition2.3} to $uv$, we have $M(G_{uv}^\prime)=m \mbox{ or } m+1$.
If $M(G_{uv}^\prime)=m+1$, then by the transformation
of Definition~\ref{definition2.1} to $ww^\prime$, we get a connected
unicyclic graph $U$ with $M(U)=m$, then by Theorem~\ref{theorem2.2} and
Theorem~\ref{theorem2.4},
$\varphi_i(G)> \varphi_i(G_{uv}^\prime)>\varphi_i(U), i=2, 3,\cdots, n-1$.
If $M(G_{uv})=m$, then by Theorem~\ref{theorem2.2},
$\varphi_i(G)> \varphi_i(G_{uv}^\prime), i=2, 3,\cdots, n-1$.

Case 2.2: When all pendent paths at $u$ have lengths $2$,
then after performing the transformation
of Definition~\ref{definition2.1} to $uv$,
by the fact $E_{uv}^u\cap M(G)=\emptyset$, we have $M(G)=M(G_{uv})$,
and $\varphi_i(G)> \varphi_i(G_{uv}), i=2, 3,\cdots, n-1$ by Theorem~\ref{theorem2.2}.

After performing Step 2,the distance between the furthest branch vertex from the cycle $C$
and $C$ will be decreased. By performing transformations in Step 2 consecutively, it turns out 
that $G_{g_1}(s_1,t_1; s_2,t_2, \cdots,s_{g_1},t_{g_1})$ has smaller signless Laplacian
coefficients in $\mathcal{G}_{g_1}(n,m)$.

Step 3: When $g(G)\geq 5$, we distinguish this case into the following two cases.

Case 3.1: If the edge $u_iu_{i+1}$ on the cycle belongs to $M(G)$,
then by the assumption of $M(G)$, $t_i=t_{i+1}=0$ holds, by performing the transformation
of Definition~\ref{definition2.5} to $u_i, u_{i+1}, u_{i+2}$, we have
$M(G)=M(G^\prime)$ and $\varphi_i(G)> \varphi_i(G^\prime), i=2, 3,\cdots, n-1$
by Theorem~\ref{theorem2.6}.

Case 3.2: If $t_i=t_{i+1}=0$, and $u_iu_{i+1}\not\in M(G)$.
This case will not occur, sice $M(G)\cup u_iu_{i+1}$ is also
a matching of $G$, a contradiction to the assumption of $M(G)$.

Case 3.3: Assume $t_{i+1}\neq 0$, $u_{i+1}u_{i+1}^\prime$ is
a pendent edge at $u_{i+1}$ and $u_{i+1}u_{i+1}^\prime\in M(G)$.
We distinguish this case into the following three cases.

Assume that $t_i\neq 0$, when $t_{i+2}=0$, by performing the transformation
of Definition~\ref{definition2.5} to $u_i, u_{i+1}, u_{i+2}$, we get $G^\prime$ with
$M(G)=M(G^\prime)$, and $\varphi_i(G)> \varphi_i(G^\prime), i=2, 3,\cdots, n-1$
by Theorem~\ref{theorem2.6}. When $t_{i+2}\neq 0$,
after performing the transformation
of Definition~\ref{definition2.7} to $u_i, u_{i+1}, u_{i+2}$, we get $G^\prime$ with
$M(G)=M(G^\prime)$, and $\varphi_i(G)> \varphi_i(G^\prime), i=2, 3,\cdots, n-1$
by Theorem~\ref{theorem2.8}.

Assume $t_i=0, s_i\neq 0$,
we perform the transformation
of Definition~\ref{definition2.1} to the edge $u_iu_{i1}$ which is incident to
$u_i$ and belongs to the pendent path of length $2$ at $u_i$.
Then
by the fact $E_{u_iu_{i1}}^{u_i}\cap M(G)=\emptyset$, we have $M(G)=M(G_{u_iu_{i1}})$,
and $\varphi_i(G)> \varphi_i(G_{u_iu_{i1}}), i=2, 3,\cdots, n-1$ by Theorem~\ref{theorem2.2}.

Assume $t_i=s_i=0$, when
$t_{i+2}\neq 0$, by performing the transformation
of Definition~\ref{definition2.5} to $u_i, u_{i+1}, u_{i+2}$, we have
$M(G)=M(G^\prime)$ and $\varphi_i(G)> \varphi_i(G^\prime), i=2, 3,\cdots, n-1$
by Theorem~\ref{theorem2.6}. When $t_{i+2}=0$, by performing the transformation
of Definition~\ref{definition2.5} to $u_i, u_{i+1}, u_{i+2}$, we have
$M(G)=M(G^\prime)+1$, then by performing the transformation
of Definition~\ref{definition2.1} to $u_iu_{i+1}$ of $G^\prime$,
we can get a connected unicyclic graph $W$ with $M(W)=M(G)$,
and $\varphi_i(G)> \varphi_i(G^\prime)> \varphi_i(W), i=2, 3,\cdots, n-1$
by Theorem~\ref{theorem2.6} and Theorem~\ref{theorem2.2}.

Therefore, after taking Step 3, we have $U^\prime \in \mathcal{G}_3(n,m)$. 
Then by Lemma~\ref{lemma3.2}, Lemma~\ref{lemma3.3}, and the discussion 
of Section 3, thus if $m=2$, $U^\prime\cong G_3^1$. 
If $n-m\leq 3$ and $6 \leq n\leq 12, m\geq 3$, $U^\prime\cong G_3^2$. 
If $n-m> 3, m\geq 3$ or $n> 12, m\geq 3$, $U^\prime\cong G_3^1$ 
or $U^\prime\cong G_3^2$, and $G_3^1, G_3^2$ are not comparable 
with respect to all the signless Laplacian coefficients.
\end{proof}

Similar to the proof of Theorem~\ref{thoerem4.1}, we have the following theorem.

\begin{theorem}\label{theorem4.2}
In the set of all n-vertex unicyclic graphs in $\mathcal{G}_{g_2}(n,m)$, then

(1). If $m=2$, then $G_4^1$ has minimal
all the signless Laplacian coefficients.

(2). If $m=3$, then $G_4^2$ has minimal
all the signless Laplacian coefficients.

(3). If $m\geq 4$, and $n=8$, 
$G_4^4$ has minimal all the signless Laplacian coefficients
$\varphi_i, i=0,1,2,\cdots,n$. 

(4). If $m\geq 4, n\geq 11$, and $n-m\geq 7$, or 
$m\geq 5, n\geq 11$, and $n-m\geq 6$, or 
$m\geq 7, n\geq 11$, and $n-m\geq 5$, then 
$G_4^2$ has minimal all the signless Laplacian coefficients
$\varphi_i, i=0,1,2,\cdots,n$. 

(5). For other cases which $n, m$ satisfy, $G_4^2$ and $G_4^4$
are two extremal graphs which have minimal signless Laplacian coefficients
$\varphi_i, i=0,1,2,\cdots,n$ in $\mathcal{G}_{g_2}(n,m)$. Furthermore,
$G_4^2$ and $G_4^4$ can not be compared.
\end{theorem}

The proof is left to the reader.

\begin{remark}
From the discussion in Section 3, the extremal graphs 
with respect to all the signless Laplacian coefficients in $\mathcal{G}_{g_1}(n,m)$
and $\mathcal{G}_{g_2}(n,m)$ can not be compared.
\end{remark}

By Theorem~\ref{theorem1.3} and Theorem~\ref{thoerem4.1}, Theorem~\ref{theorem4.2},
we obtained the following two corollaries.

\begin{corollary}\label{corollary4.3}
If $m=2$, then for $G\in \mathcal{G}_{g_1}(n,m)$,
we have $IE(G)\geq IE(G_3^1)$.
with equality if and only if $G\cong G_3^1$. 
If $m\geq3$, then for $G\in \mathcal{G}_{g_1}(n,m)$, we have $IE(G)\geq\mbox{min}\{IE(G_3^1), IE(G_3^2)\}$, 
with equality only if $G\cong G_3^1$ or $G\cong G_3^2$.
\end{corollary}

\begin{corollary}\label{corollary4.4}
If $m=2$, then for $G\in \mathcal{G}_{g_2}(n,m)$, 
we have $IE(G)\geq IE(G_4^1)$. 
with equality if and only if $G\cong G_4^1$. 
If $m=3$, then for $G\in \mathcal{G}_{g_2}(n,m)$,
we have $IE(G)\geq IE(G_4^2)$.
with equality if and only if $G\cong G_4^2$. 
If $m\geq 4$, then for $G\in \mathcal{G}_{g_2}(n,m)$,
we have $IE(G)\geq\mbox{min}\{IE(G_4^2), IE(G_4^4)\}$,
with equality only if $G\cong G_4^2$ or $G\cong G_4^4$.
\end{corollary}

\section{\large\bf{The Signless Laplacian coefficients of unicyclic graphs in $\mathcal{G}(n)$}}
For convenience, denote $G_3(0,0;0,0;0,n-3)$ as $S_3^\prime$, 
and denote $G_4(0,0;0,0;0,0;0,n-4)$ as $S_4^\prime$.

\begin{theorem}\label{thoerem5.1}
In the set of all n-vertex unicyclic graphs in $\mathcal{G}_{g_1}(n)$,
$S_3^\prime$ has minimum signless Laplacian coefficients
$\varphi_i, i=0,1,2,\cdots,n-1,n$.
\end{theorem}

\begin{proof}
For $n=5$, $S_3^\prime$ minimize
all the signless Laplacian coefficients. This result can be obtained by
using Matlab 7.0.

For $n\geq 6$, let $G$ be an arbitrary graph in $\mathcal{G}_{g_1}(n)$.
We need to prove after series of transformations, $G$ will
become to $S_3^\prime$ and $S_3^\prime$ has minimum signless Laplacian coefficients
in $\mathcal{G}_{g_1}(n)$.

Step 1: When there is a non-pendent edge $uv$ which is not on the cycle.
By performing the transformation of Definition~\ref{definition2.1} to $uv$,
we have $G_{uv}\in\mathcal{G}_{g_1}(n)$, and
$\varphi_i(G)>\varphi_i(G_{uv})$, $i=2,3,\cdots,n-1$ by Theorem~\ref{theorem2.2}.

After performing Step 1 consecutively, we
have that $G_{g_1}(0,t_1;0,t_2; \cdots;0,t_{g_1})$ has smaller signless Laplacian
coefficients in $\mathcal{G}_{g_1}(n)$.

Step 2: When $g(G)\geq 5$, by performing the transformation
of Definition~\ref{definition2.5} to $u_i, u_{i+1}, u_{i+2}$, we have
$\varphi_i(G)> \varphi_i(G^\prime), i=2, 3,\cdots, n-1$
by Theorem~\ref{theorem2.6}. Then by
performing the transformation of Definition~\ref{definition2.1}
to the new pendent edge $u_{i+1}u_{i+2}$,
a new graph with smaller signless Laplacian coefficients is obtained by Theorem~\ref{theorem2.2}.

Therefore, after taking Step 2 consecutively, 
we obtain the extremal graph with
minimal signless Laplacian coefficients which must belong to $\mathcal{G}_3(n)$.
Then by Lemma~\ref{lemma3.1}, Lemma~\ref{lemma3.4}, and the discussion
of Section 3, thus $S_3^\prime$ is the resulting graph which has
minimum signless Laplacian coefficients in $\mathcal{G}_{g_1}(n)$. 
\end{proof}

\begin{theorem}\label{thoerem5.2}
In the set of all n-vertex unicyclic graphs in $\mathcal{G}_{g_2}(n)$,
$S_4^\prime$ has minimum signless Laplacian coefficients
$\varphi_i, i=0,1,2,\cdots,n-1,n$.
\end{theorem}

\begin{proof}
For $n=5$, $S_4^\prime$ is the unique graph $\mathcal{G}_{g_2}(5)$,
so it has the minimum
signless Laplacian coefficients in $\mathcal{G}_{g_2}(5)$.

For $n\geq 6$, let $G$ be an arbitrary graph in $\mathcal{G}_{g_2}(n)$.
The method in the proof of Theorem~\ref{thoerem5.1} can be used to
prove this result similarly.
\end{proof}

\begin{remark}
From the discussion in Section 3, the extremal graphs $S_3^\prime$ and $S_4^\prime$ 
with respect to all the signless Laplacian coefficients in $\mathcal{G}_{g_1}(n)$
and $\mathcal{G}_{g_2}(n)$ can not be compared.
\end{remark}

By Theorem~\ref{theorem1.3} and Theorem~\ref{thoerem5.1}, Theorem~\ref{thoerem5.2},
we obtained the following two corollaries.

\begin{corollary}\label{corollary5.3}
In the set of all n-vertex unicyclic graphs in $\mathcal{G}_{g_1}(n)$,
$S_3^\prime$ is the unique graph with the minimal $IE$.
\end{corollary}

\begin{corollary}\label{corollary5.4}
In the set of all n-vertex unicyclic graphs in $\mathcal{G}_{g_2}(n)$,
$S_4^\prime$ is the unique graph with the minimal $IE$.
\end{corollary}

\end{document}